\documentclass{article}
\usepackage{amsmath,natbib}
\usepackage{amsfonts}
\usepackage{color}
\usepackage{comment}
\usepackage{appendix}
\newtheorem{theorem}{Theorem}

\newtheorem{assumption}[theorem]{Assumption}

\newtheorem{definition}[theorem]{Definition}

\newtheorem{lemma}[theorem]{Lemma}

\newtheorem{proposition}[theorem]{Proposition}
\newtheorem{remark}[theorem]{Remark}

\newenvironment{proof}[1][Proof]{\noindent\textbf{#1.} }{\ \rule{0.5em}{0.5em}}
\begin{document}

\title{A Mathematical Framework for Linear Response Theory for Nonautonomous Systems}
\author{
    Stefano Galatolo$^{1}$\footnote{Email:\texttt{stefano.galatolo@unipi.it}}
    \quad Valerio Lucarini$^{2,3}$\footnote{Email:\texttt{v.lucarini@leicester.ac.uk}} \\[2ex]
    $^{1}$Dipartimento di Matematica, Universit\`a di Pisa, Pisa, Italy\\
    $^{2}$School of Computing and Mathematical Sciences\\ University of Leicester, Leicester, UK\\
       $^{3}$Business School, Imperial College London, London, UK
}

\maketitle

\begin{abstract}
Linear Response theory aims to predict how added forcing alters the statistical properties of a system. 
This has been predominantly studied 
for autonomous dynamical systems, yet many systems in the physical, natural, and social sciences are inherently nonautonomous. 
In such settings, one would like to understand how the system’s time-dependent statistical properties change when additional infinitesimal forcings are applied. This question is of clear practical relevance, but from a rigorous mathematical viewpoint it has been addressed only for a few specific classes of systems/perturbations.
Here we provide a rigorous linear response theory for a rather general class of deterministic and random nonautonomous systems satisfying a specific set of assumptions that are the counterpart of 
the standard assumptions used in the autonomous setting.  A central ingredient is rapid loss of memory 
along the nonautonomous evolution.  
 Our main strategy is to reformulate the sequential dynamics as a fixed-point problem for a global transfer operator acting on an extended sequence space of measures. This yields explicit and readily implementable causal response formulas for predicting the effect of small perturbations on time-dependent statistical states.
 We illustrate the theory on two representative classes: sequential compositions of  expanding maps and sequential compositions of noisy random maps. 
 We also establish linear response for reflected
Euler-Maruyama discretizations of dissipative
nonautonomous stochastic differential equations with
nondegenerate additive Gaussian noise, on fixed compact
domains. As a physical application, we study the response
of a Ghil-Sellers climate energy balance model,
discretized in latitude and time to a
time-dependent perturbation of the greenhouse parameter.
\end{abstract}
\section{Introduction}
Investigating quantitatively how a system responds to perturbations is a key challenge in mathematics, the natural and quantitative social sciences, and engineering. In the context of dynamical systems and their statistical properties, the problem can be cast as follows: how do the statistical properties of a system change when we perturb the reference dynamics with a small additional forcing? Is it possible to predict the response from the knowledge of the acting forcing and of the statistical properties of the unperturbed state? 
In this paper, we will set up a framework where one can answer these questions in the context of non-autonomous,  time-dependent systems.

\paragraph{Literature Review: the Physical and the Mathematical Viewpoints} In the context of time-independent deterministic or random systems, a vast body of literature exists on this topic. 

 In the \emph{physical literature}, the starting point can be traced to { Einstein who investigated the link between microscopic fluctuations and macroscopic drag in his 1905 groundbreaking work on  the Brownian motion \cite{BrownianMotionEinstein}. These findings were extended and systematized in the context of the fluctuation-dissipation theorem (FDT), which establishes a link between unforced fluctuations of a system and its response to perturbations~\cite{Kubo.1966}}.
The linear response of the system can be predicted {by convoluting the time modulation for the forcing with suitably defined causal Green's functions. Such Green's functions}  have particularly intuitive expressions found for systems near thermodynamic equilibrium \cite{Kubo1957}. Well-established applications of linear response theory include solid state physics and optics \cite{Lucarini2005}, plasma physics \cite{Turski1969}, and stellar dynamics~\cite{BinneyTremaine2008}, neuroscience \cite{Cessac2019}, multiagent models \cite{LPZ2020,LPZ2021}, as well as in many other systems, both near and far from thermodynamic equilibrium   \cite{Hanggi1982,Ottinger2005,Marconi2008,Baiesi2013,Colangeli_2012,Colangeli_2014,Lucarini2017,LucariniChekroun2023,LucariniChekroun2024}; {see also the recently dedicated special issue \cite{Gottwald2020LRT}, the dedicated review papers \cite{Marconi2008,Sarracino2019}, and book \cite{Sarracino2025}.}

{Very recently, considerable scientific advances have come, both in terms of theoretical development and of practical use, from the combination of response theory with data driven methods. Specifically, combining response theory with Koopmanism \cite{Budisic2012,KutzBrunton2016,Brunton2022,colbrook2024} leads, under suitable hypotheses, to the interpretability of the response operators, as they can be written as the sum of terms, each associated with a specific mode of unforced variability of the system \cite{Santos2022}. This approach has the great advantage of allowing the construction of response operators from the properties of the unperturbed system and of identifying which modes provide the most important contribution to the response \cite{lucarini2025generalframeworklinkingfree,zagli_SIAM:2026} and can be shown to have a clear connection with Markov state modeling \cite{Lucarini2025}. A separate line of fruitful investigation comes from combining response theory with generative modeling \cite{giorgini2024linear,giorgini2025statistical,giorgini2024datadriven}, providing encouraging results in terms of constructing response operators for high-dimensional systems.}

A common limitation of most of the response results mentioned so far, when applied to real-world phenomena, is that they consider perturbations of a steady reference state associated with an invariant measure.
But, indeed, many systems of interest evolve according to dynamical laws that are explicitly time-dependent, and, yet, we would like to be able to study the response of such  systems to perturbations. The time modulation can be periodic - think of any ecological system influenced by the seasonal cycle - or aperiodic - think of the climate, which is impacted by irregular geochemical forcings, by human activities, on top of more regular astrophysical and astronomical modulations {(\cite{Crisan2026,Ashwin2026})}. Additionally, explicit time-dependence emerges each time we consider a multiscale system and focus our analysis on a restricted range of time scales, see \cite{santos2021} and the physically inspiring discussion in \cite{saltzman_dynamical}.

The extension of fluctuation-dissipation results and linear response to aging systems have been presented since a while in the physical literature, see e.g. \cite{Crisanti2003,Bertier2007}. Yet aging systems are a very special class of system whose statistical properties change with time. In the context of diffusion processes, Branicki and Uda \cite{Branicki2021} have provided a rigorous generalization of response theory for systems whose reference dynamics is periodic. 

The recent paper \cite{Lucarini2026}  presented on more heuristic basis general response formulas for a rather general class of time-dependent systems, providing explicit results  for time dependent finite Markov chains and diffusion processes. That work  
also shows numerical evidence of the validity of the proposed nonautonomous linear response formula on a simple yet foundational climate model, namely the celebrated Ghil-Sellers model \cite{Sellers,Ghil1976,Bodai2015}. 
The paper  \cite{Lucarini2026} also provides a rather extensive discussion  on motivations and challenges of nonautonomous response theory,
together with an outline of a mathematical framework. A central challenge is to describe the response
of equivariant measures associated with pullback attractors
\cite{Crauel1997,CHEKROUN20111685}.


 In the {\em mathematical literature}, in the context of \emph{deterministic dynamical systems}, Ruelle initiated the mathematical investigation of linear response for Axiom A chaotic systems, provided explicit response formulas {both for static and time-dependent perturbations}, and clarified that the classical form of the fluctuation-dissipation theorem does not hold because of the different geometrical properties of the stable vs unstable tangent space \cite{Ruelle1997DifferentiationSRBstates,Ruelle1998GeneralLinearresponseformula,ruelle_nonequilibrium_1998,ruelle_review_2009}. This leads to major difficulties in directly implementing Ruelle's formulas \cite{abramov2007}, even if in recent years much progress has been achieved using adjoint and shadowing methods \cite{Wang2013,Chandramoorthy2020,Ni2020,Chandramoorthy2022,Ni2023,Ni2026FastDifferentiationHyperbolicChaos}.

Ruelle's results have then been first generalized to some cases of partially hyperbolic systems \cite{Dolgopyat2004}, and, later on, to many other classes of systems (e.g. \cite{liverani2006,BL07,B08,BaladiTodd2016,Korepanov_2016,Galatolo2022,SelleyTanzi2021,Bah_Gal_2024,Baladi_2008}) also extending the methods  in more abstract terms by taking advantage of the transfer operator formalism \cite{B00}. 
Note that transfer operator methods can also be implemented for the numerical computation of linear response, and can give very precise computations, also with explicit bounds on the approximation errors \cite{Pollicott_2016, FroylandPhalempin2025OptimalLinearResponse,Bahsoun_2018}  however these methods, typically suffer from the curse of dimensionality and are not efficient for performing actual calculations of the response of high-dimensional systems. 

Response theory is easier to justify in rigorous mathematical terms when considering \emph{random dynamical systems} and diffusion processes  \cite{HairerMajda2010, DemboDeuschel2010,pavliotisbook2014,Galatolo_2019,BAHSOUN20,GaSe}.
Still related to the random dynamical  case, we also mention integral analogs of the perturbed Kolmogorov equation \cite{Kartashov2005} and finite state Markov chains \cite{Lucarini2016,Antown2018,Santos2020}, which have the great advantage of providing an equation-agnostic platform \cite{Lucarini2025}  and can be coupled naturally with data-driven reduced-order techniques such as Markov state modeling \cite{Pande2010,Husic2019,GalatoloLucariniNisoli2026}.

Rigorous quenched linear response results for random dynamical
systems are particularly close to the setting of this paper:
they describe the response of equivariant families along
almost every fixed realization of a sequence of random maps.
In the usual formulation, these maps are selected by a fixed
stationary stochastic process, represented by a
probability-preserving base dynamics.
This provides a sequential perspective, with an additional
stationary structure that is not required in our framework.
Results in this direction  have  been provided in \cite{DragicevicGiuliettiSedro2022}, where an abstract 
response theory for parametrized transfer-operator cocycles is developed and applied 
to smooth expanding-on-average cocycles on the circle, while \cite{DragicevicSedro2021} treated random
products of nearby Anosov diffeomorphisms in a uniformly hyperbolic setting. 
\cite{CrimminsNakano2024} develop a spectral approach to quenched linear and higher-order response under uniform mixing assumptions, with applications to random Anosov diffeomorphisms and random  extensions of expanding maps. Their approach lifts the transfer-operator cocycle to an operator on a space of measurable families, providing a closely related perspective to the global-operator formulation used here. 
\cite{DragicevicHafouta2026} establish effective quenched linear response for classes of non-uniformly expanding random systems, with moment bounds on the random constants controlling the response remainder. Under suitable additional assumptions, these estimates yield annealed linear response and differentiability of the asymptotic variance in the central limit theorem.

Cone methods are particularly well suited to sequential dynamics:
uniform contraction in a projective metric, possibly after a fixed
number of steps, provides control over compositions of different
transfer operators. \cite{SedroRugh2021} uses this
approach to establish quenched linear response for random
compositions of uniformly expanding maps.
In Remark~1.4, they  indicate how their method and
response formula can be adapted to sequential dynamics,
although this extension is not developed in detail.
The methods developed for random hyperbolic systems in \cite{DragicevicSedro2021}
also appear amenable to adaptation to sequential dynamics,
although this case is not explicitly treated there.

Bridging the quenched and sequential settings,
\cite{DragicevicGonzalezTokmanSedro2024} establish linear
response for quenched and sequential compositions of intermittent
LSV-type maps using cone techniques. In a related direction 
\cite{DemersLiverani2025CLT} develop a framework based on projective
metrics on complex cones to prove central limit theorems with
error bounds, with applications to sequential expanding maps
and dispersing billiards.


\paragraph{Our contribution}

Motivated in part by the formal response formulas proposed
in \cite{Lucarini2026}, we develop an abstract operator-theoretic
framework for linear response in discrete-time deterministic
and random nonautonomous systems.
We consider arbitrary prescribed sequences of transfer operators
satisfying uniform assumptions, without requiring a
probability-preserving base dynamics.
Passing to an extended space of sequences of measures provides
a natural way to characterize their equivariant families as
fixed points of a single global transfer operator.
The key ingredient is uniform exponential loss of memory,
which allows us to construct a resolvent for the global
transfer operator on sequences of zero-mass measures.
This resolvent lets us invert the linearized fixed-point
equation and express the response as a convergent sum of
contributions from past perturbations.
Together with uniform regularity and differentiability assumptions,
which parallel the standard autonomous hypotheses, this yields
a linear response result that is uniform in time (Theorem~\ref{thm:LR_strong}).

Our framework accommodates aperiodic reference dynamics,
complementing the results for time-periodic diffusion processes
in \cite{Branicki2021}.
We apply this result to sequential compositions of expanding
maps and to sequential random maps with additive noise.
 This latter general class of systems includes models of many important physical systems.
For this case, a uniform Doeblin condition provides
the required exponential loss of memory.

In particular, we apply the framework to Euler-Maruyama
discretizations of dissipative nonautonomous stochastic
differential equations with nondegenerate additive Gaussian
noise. Euler-Maruyama is a standard method for numerical
simulations of stochastic differential equations and provides
a natural discrete-time approximation of their dynamics
\cite{Kloeden2011}.

The flexibility of the framework allows us to apply it
to a physically motivated climate model: the stochastic
Ghil-Sellers energy balance model \cite{Ghil2020}.
This model combines ice-albedo feedback,
greenhouse-modulated outgoing radiation, and meridional
heat transport \cite{Sellers,Ghil1976,Bodai2015}.
For its discretization in latitude and in time, with reflection
on a fixed compact temperature domain, we prove linear
response of the equivariant family to a time-dependent
perturbation of the greenhouse parameter representing
increased CO$_2$ forcing.

The paper is organized as follows. Section~\ref{s:sequential} introduces the main framework: suitable sequence spaces, the associated global transfer operator, and equivariant measures realized as fixed points (equivalently, eigenvectors for the unit eigenvalue) of this operator. Section~\ref{lom} presents the key dynamical hypothesis underpinning the theory, namely loss of memory. In Section~\ref{linearresponse} we establish resolvent estimates for the global map and prove the abstract linear response result for time-dependent systems, obtaining explicit response formulas. Sections~\ref{sequential} and \ref{sec:LR_random_noise} apply the theory to  representative classes of sequential systems, both in the deterministic case  (expanding maps) and in the random case (additive-noise dynamics),
{ including the application to Euler-Maruyama discretizations of SDE}. {Section \ref{subsec:LR_Ghil_Sellers} focuses on the above-mentioned case of the Ghil-Sellers climate model.}
Section~\ref{conclusions} discusses the results and outlines directions for future work. Appendices~\ref{appendixa} and \ref{subsec:mixed_continuity_n_branches} collect technical tools used in the proofs.

\section{Sequential nonautonomous systems, sequence spaces, and global map}\label{s:sequential}

We address the questions described in the introduction   using a transfer-operator approach. Since we also aim to cover deterministic systems, following the standard practice in the literature, we work with transfer operators acting on two Banach spaces, a strong and a weak one. 
We recast the nonautonomous setting as an autonomous one on an extended phase space, namely a suitable sequence space. In this formulation, equivariant families become fixed points of an associated (global) transfer operator. Passing to the sequence space restores autonomy, at the price of losing the compactness of the strong-to-weak embedding in the corresponding sequence spaces. In this section we introduce this construction and the required preliminary notions.

Let $X$ be separable metric space.
Let $(B_s,\|\cdot\|_s)$ be a Banach space of finite signed Borel measures on $X$.
Let us also consider a "weaker" Banach space of signed measures $(B_w,\|\cdot\|_w)$ such that $B_s\subseteq B_w$ and $\|\cdot\|_w\leq \|\cdot\|_s$.
\footnote{In concrete examples, when dealing with expanding maps or noisy systems  $B_s$ can be a space of measures having regular densities, as for example the Sobolev space $W^{1,1}$, while $B_w$ usually is set to be $L^1$ (see Section \ref{sequential}). Dealing with hyperbolic dynamical systems, in the presence of contracting directions, the framework can be extended to suitable "anisotropic" distribution spaces (see e.g. \cite{liverani2006}).}
Write $\mu(X)$ for the total mass and set
\[
V_s:=\{\mu\in B_s:\mu(X)=0\}.
\]
We will also suppose that $V_s$ is a closed subspace of $B_s$.
Define the sequence space
\[
\mathcal B_s := \ell^\infty(\mathbb Z;B_s),
\qquad
\|\boldsymbol\mu\|_{\mathcal B_s}:=\sup_{n\in\mathbb Z}\|\mu_n\|_s,
\]
and the zero-mass subspace
\[
\mathcal V_s := \ell^\infty(\mathbb Z;V_s)\subset \mathcal B_s.
\]


Let $\mathbb S:\mathcal B_s\to\mathcal B_s$ denote the left shift,
\[
(\mathbb S\boldsymbol\mu)_n := \mu_{n+1}.
\]
Then $\mathbb S$ is a linear isometry on $\mathcal B_s$.

Now we introduce the time dependent dynamics by considering a sequence of transfer operators. To formalize the idea of a perturbation in the family, we will introduce a further parameter $\varepsilon$ to adjust the strength of the perturbation. 
The operators for $\varepsilon=0$ will then be considered as unperturbed ones.

Let $\epsilon_0>0$ and $(L_n^\varepsilon)_{n\in\mathbb Z}$ be a family of uniformly bounded linear operators
\[
L_n^\varepsilon:B_s\to B_s,\qquad L_n^\varepsilon:B_w\to B_w \qquad \varepsilon\in [0, \varepsilon_0),
\]
which are \emph{Markov} in the sense that they are positive and they preserve total mass:
\begin{equation}\label{eq:mass_preserving}
(L_n^\varepsilon\mu)(X)=\mu(X)\qquad\forall \mu\in B_s,\ \forall n,\ \forall \varepsilon \in [0, \varepsilon_0).
\end{equation}
We define the (parameter-dependent) \emph{global map} $\mathbb F_\varepsilon:\mathcal B_s\to\mathcal B_s$ by
\[
(\mathbb F_\varepsilon(\boldsymbol\mu))_n := L_n^\varepsilon \mu_n.
\]
A sequence of probability measures $\boldsymbol\mu^\varepsilon=(\mu_n^\varepsilon)_{n\in\mathbb Z}$ in ${\mathcal B_s}$ is said to be an \emph{equivariant family} for $(L_n^\varepsilon)$ if
\begin{equation}\label{eq:equiv}
\mu_{n+1}^\varepsilon = L_n^\varepsilon \mu_n^\varepsilon\qquad\forall n\in\mathbb Z,
\end{equation}
which is equivalent to the fixed point equation on $\mathcal B_s$
\begin{equation}\label{eq:FP_shift}
\mathbb S\boldsymbol\mu^\varepsilon = \mathbb F_\varepsilon(\boldsymbol\mu^\varepsilon).
\end{equation}

The equivariant family plays the role of an invariant
measure in the nonautonomous setting. In the applications
below, it is selected by pullback limits of evolving
probability measures.

It is convenient to work on another global map that is the composition
\[
\mathbb T_\varepsilon := \mathbb S^{-1}\mathbb F_\varepsilon,
\qquad
(\mathbb T_\varepsilon\boldsymbol v)_n = L_{n-1}^\varepsilon v_{n-1}.
\]
Then \eqref{eq:FP_shift} is equivalent to
\begin{equation}\label{eq:FP_T}
\boldsymbol\mu^\varepsilon = \mathbb T_\varepsilon \boldsymbol\mu^\varepsilon.
\end{equation}



Note that in what follows we will consider families of transfer operators parametrized by $\varepsilon$, where $\varepsilon=0$ describes our reference state, whilst considering $\varepsilon\neq0$ is associated with perturbed states. 

\section{Loss of memory} \label{lom}

In the autonomous setting, linear response results are often proved under a spectral gap assumption or some kind of fast convergence to equilibrium (with notable exceptions such as \cite{GalatoloSorrentino2022, Korepanov_2016}). In our nonautonomous framework, the analogous role is played by loss of memory. This concept quantifies the speed at which a non autonomous dynamical system forgets its initial condition, and this is also a key ingredient proposed in \cite{Lucarini2026}. In this section, we introduce the concept and derive the first fundamental properties.

\begin{definition}[Loss of memory on the strong space]\label{def:lom}
We say that $(L_n)_{n\in\mathbb Z}$ has a strong \emph{loss of memory (SLoM) on $V_s$} if there exists a sequence
$a(k)\downarrow 0$ such that for all $j\in\mathbb Z$, all $k\ge 1$, and all $v\in V_s$,
\begin{equation}\label{eq:lom_general}
\bigl\|L^{(j,j+k-1)} v\bigr\|_s \;\le\; a(k)\,\|v\|_s 
\end{equation}
where $L^{(j,j+k-1)}:=L_{(j+k-1)}\circ...\circ L_j$.
\end{definition}

\begin{remark}The uniformity in the initial time $j$ implies that strong
loss of memory is necessarily exponential: the sequence
$a(k)$ in \eqref{eq:lom_general} can be replaced by
$C\rho^k$ for suitable $C>0$ and $\rho\in(0,1)$.
Indeed, the quantities
\[
b(k):=\sup_{j\in\mathbb Z}
\|L_{(j,j+k-1)}\|_{V_s\to V_s}
\]
are submultiplicative and converge to zero.
Furthermore, since differences of probability measures have zero mass, \eqref{eq:lom_general} implies
\[
\|L^{(j,j+k-1)}\nu - L^{(j,j+k-1)}\nu'\|_s \le a(k)\|\nu-\nu'\|_s
\quad\forall \nu,\nu'\in\mathcal P_s.
\]
Where 
$$P_s := \{\mu \in B_s : \mu \geq 0,\ \mu(X)=1\}.$$
\end{remark}

The presence of loss of memory ensure the uniqueness of equivariant measures in the strong space.

\begin{proposition}[Uniqueness of the fixed point in $\ell^\infty$]\label{prop:unique_fixedpoint}
Assume that $(L^{\varepsilon}_n)$ has loss of memory on $V_s$ in the sense of \eqref{eq:lom_general}.
Then $\mathbb T_\varepsilon$ has \emph{at most one} fixed point in $\ell^\infty(\mathbb Z;B_s)$ which is a sequence of elements of $P_s$.
Equivalently, there exists at most one bounded equivariant family  $\boldsymbol\mu$ of probability measures $\mu_n$ with
$\sup_n\|\mu_n\|_s<\infty$.
\end{proposition}

\begin{proof}
Let $\boldsymbol\mu,\boldsymbol\nu\in\ell^\infty(\mathbb Z;B_s)$ be two equivariant families and fixed  points of $\mathbb T_\varepsilon$.
Set $\Delta_n:=\mu_n-\nu_n$. Then $\Delta_n\in V_s$ for all $n$ and
\[
\Delta_n=\mu_n-\nu_n=L_{n-1}\mu_{n-1}-L_{n-1}\nu_{n-1}=L_{n-1}\Delta_{n-1}.
\]
Iterating $k$ steps backwards yields
\[
\Delta_n = L^{(n-k,n-1)}\Delta_{n-k}.
\]
Taking $\|\cdot\|_s$ and applying \eqref{eq:lom_general} gives
\[
\|\Delta_n\|_s \le a(k)\,\|\Delta_{n-k}\|_s \le a(k)\,\sup_{m\in\mathbb Z}\|\Delta_m\|_s.
\]
Letting $k\to\infty$ and using $a(k)\to 0$ yields $\|\Delta_n\|_s=0$ for every $n$, hence $\boldsymbol\mu=\boldsymbol\nu$.
\end{proof}

\begin{remark}[Existence via pullback limits]
In applications, one typically combines SLoM  with uniform strong bounds coming from a Lasota--Yorke inequality
to show existence of a bounded fixed point, e.g. by defining
\(
\mu_n:=\lim_{k\to\infty} L^{(n-k,n-1)}\nu
\)
for some $\nu\in\mathcal P_s$ and proving convergence in $\|\cdot\|_s$ using SLoM on $V_s$.
Proposition~\ref{prop:unique_fixedpoint} then yields uniqueness in the bounded class.
\end{remark}



\section{Linear response}\label{linearresponse}
In this section we prove a linear response statement for a sequential family of transfer operators as in the previous setting.
We will  consider $\varepsilon_0>0$ and for $\varepsilon\in [0,\varepsilon_0)$ we consider a sequential family of operators $L^\varepsilon _n$ with associated global maps $\mathbb T_\varepsilon$ having  equivariant measures $\boldsymbol\mu^\varepsilon$. 
Throughout this section, we use the abbreviations
$\boldsymbol\mu:=\boldsymbol\mu^0$, $\mu_n:=\mu_n^0$, $L_n:=L_n^0$,
and $\mathbb{F}:=\mathbb{F}_0$ for the unperturbed system.
We will also assume the following:

\begin{assumption}[Uniform strong bounds for the equivariant family]\label{ass:mu_bound}
For each $\varepsilon\in [0,\varepsilon_0)$ the family of operators $L_n^\varepsilon$ has an equivariant family $\boldsymbol\mu^\varepsilon\in \mathcal B_s$ and
\begin{equation}\label{eq:mu_uniform_bound}
\sup_{\varepsilon\in[0,\varepsilon_0)}\ \|\boldsymbol\mu^\varepsilon\|_{\mathcal B_s}
= \sup_{\varepsilon\in [0,\varepsilon_0)}\ \sup_{n\in\mathbb Z}\|\mu_n^\varepsilon\|_s
\le M_s<\infty.
\end{equation}
\end{assumption}

\begin{assumption}[Strong differentiability of the cocycle]\label{ass:strong_derivative}
There exists a sequence of elements of $B_s$ which we will denote as $(\dot L_n \mu^0_n)_{n\in\mathbb Z} \in B_s$, such that
\begin{equation}\label{eq:diff_unif}
\lim_{\varepsilon\to0}\ \sup_{n\in\mathbb Z}
\left\|
\frac{L_n^\varepsilon \mu^0_n- L_n^0\mu^0_n}{\varepsilon} - \dot L_n \mu^0_n
\right\|_{s} = 0,
\end{equation}
furthermore suppose that 
\begin{equation}\label{eq:dotL_bound}
\sup_{n\in\mathbb Z}\|\dot L_n\mu^0_n\|_{s} < \infty
\end{equation}
and  
\begin{equation}
\lim_{\varepsilon\to0}\sup_{n\in\mathbb Z, v\in \{B_s | \ ||v||_s=1\}}\| (L^\varepsilon_n -L^0_n)v\|_{w}=0.
\end{equation}
\end{assumption}
Loosely speaking, Assumption~\ref{ass:strong_derivative} means that, along the reference family $(\mu^0_n)$, the cocycle admits a first-order expansion in $\varepsilon$, namely
\[
L_n^\varepsilon \mu^0_n \;=\; L_n^0 \mu^0_n \;+\; \varepsilon\,\dot L_n \mu^0_n \;+\; o(\varepsilon)
\quad\text{in }B_s,\ \text{uniformly in }n,
\]
so that one may heuristically view $L_n^\varepsilon \approx L_n^0+\varepsilon\,\dot L_n$ at first order.
We also note that, depending on the perturbation, establishing strong convergence in \eqref{eq:diff_unif} may require the reference densities $\mu_n$ to have one additional degree of regularity beyond Assumption~\ref{ass:mu_bound}; see Section~\ref{sequential} for a concrete instance of this issue and how it can be handled.

\begin{assumption}[Strong exponential loss of memory in $B_s$]\label{ass:uniform_ELoM}
There exist constants $C\ge1$ and $\rho\in(0,1)$ such that for every $\varepsilon\in [0,\varepsilon_0)$, every $m<n$,
and every $v\in V_s$,
\begin{equation}\label{eq:ELoM_eps}
\bigl\|L_{n-1}^\varepsilon \cdots L_m^\varepsilon\, v\bigr\|_s
\le C\rho^{\,n-m}\|v\|_s.
\end{equation}
\end{assumption}

\begin{remark}
Assumption~\ref{ass:mu_bound} is typically obtained from a uniform Lasota-Yorke inequality in $\varepsilon$ (see Appendix \ref{appendixa} for examples on how to obtain this in the deterministic and random setting). In this section we take it as a standing hypothesis.
We also remark that Assumption~\ref{ass:mu_bound}, combined with Assumption~\ref{ass:uniform_ELoM} imply that the pushforward of any measure in $B_s$ has bounded $B_s$ norm.

\end{remark}


\subsection{Existence of the resolvent  on the zero-mass subspace}

In this section we show a kind of Neumann-series  based invertibility of $I-\mathbb T_\varepsilon$ on $\mathcal V_s$ and investigate its basic properties, which will be used to establish our Linear Response result.

\begin{lemma}[Neumann series for the global cocycle]\label{lem:neumann}
Under Assumptions~\ref{ass:uniform_ELoM} and \eqref{eq:mass_preserving}, for each $\varepsilon \in[0,\varepsilon_0)$
the operator $\mathbb T_\varepsilon$ leaves $\mathcal V_s$ invariant and satisfies
\begin{equation}\label{eq:T_power_bound}
\|\mathbb T_\varepsilon^k\|_{\mathcal V_s\to \mathcal V_s} \le C\rho^k\qquad\forall k\ge 1.
\end{equation}
In particular, $I-\mathbb T_\varepsilon$ is invertible on $\mathcal V_s$ and
\begin{equation}\label{eq:resolvent_neumann}
(I-\mathbb T_\varepsilon)^{-1} = \sum_{k=0}^\infty \mathbb T_\varepsilon^k
\quad\text{(convergence in operator norm on $\mathcal V_s$)},
\end{equation}
with the bound
\begin{equation}\label{eq:resolvent_bound}
\|(I-\mathbb T_\varepsilon)^{-1}\|_{\mathcal V_s\to\mathcal V_s} \le \frac{C}{1-\rho}.
\end{equation}
\end{lemma}

\begin{proof}
Mass preservation \eqref{eq:mass_preserving} implies $L_n^\varepsilon(V_s)\subset V_s$; hence $\mathbb T_\varepsilon(\mathcal V_s)\subset\mathcal V_s$.
Iterating $(\mathbb T_\varepsilon\boldsymbol v)_n=L_{n-1}^\varepsilon v_{n-1}$ yields
\[
(\mathbb T_\varepsilon^k \boldsymbol v)_n = L_{n-1}^\varepsilon \cdots L_{n-k}^\varepsilon\, v_{n-k}.
\]
Since $v_{n-k}\in V_s$, Assumption~\ref{ass:uniform_ELoM} gives
\[
\|(\mathbb T_\varepsilon^k \boldsymbol v)_n\|_s \le C\rho^k \|v_{n-k}\|_s \le C\rho^k \|\boldsymbol v\|_{\mathcal B_s}.
\]
Taking $\sup_n$ yields \eqref{eq:T_power_bound}. Then $\sum_{k\ge0}\mathbb T_\varepsilon^k$ converges in operator norm on $\mathcal V_s$
(because $\sum_{k\ge0}C\rho^k<\infty$), and equals $(I-\mathbb T_\varepsilon)^{-1}$ by the standard Neumann-series argument, giving
\eqref{eq:resolvent_neumann} and \eqref{eq:resolvent_bound}.
\end{proof}


The previous lemma shows the existence of the resolvent operator $(I-\mathbb T_\varepsilon)^{-1}$, in the next lemma we prove its stability under perturbation.
The approach is general and axiomatic, however almost all the asumptions we will consider for $\mathbb T_\varepsilon$ are direct consequences of the  Assumptions ~\ref{ass:mu_bound}--\ref{ass:uniform_ELoM} on the operators $L_n^\varepsilon$ listed in the previous section.

\begin{lemma}[Quantitative continuity of the resolvent in mixed norm ]\label{lem:resolvent_continuity_mixed}
Let $B_s\hookrightarrow B_w$ be Banach spaces with continuous embedding $\|u\|_w\le \|u\|_s$.
Let us denote, as before $\mathcal B_s=\ell^\infty(\mathbb Z;B_s)$ and $\mathcal B_w=\ell^\infty(\mathbb Z;B_w)$, and let
$\mathcal V_s\subset\mathcal B_s$, $\mathcal V_w\subset\mathcal B_w$ be the zero-mass subspaces.

Let $\varepsilon \in[0,\varepsilon_0)$ and $\mathbb T_0,\mathbb T_\varepsilon$ be linear operators such that:
\begin{enumerate}
\item[(A1)] (\emph{Strong exponential memory loss on }\ $\mathcal V_s$)
There exist $C\ge 1$ and $\rho\in(0,1)$ such that for all $k\ge 1$ and all $v\in\mathcal V_s$, $\varepsilon\in[0,\varepsilon_0)$
\[
\|\mathbb T_\varepsilon^k v\|_{\mathcal B_s}\le C\rho^k\|v\|_{\mathcal B_s}.
\]
\item[(A2)] (\emph{Weak power-boundedness on }\ $\mathcal V_w$)
There exists $M_w\ge 1$ such that for all $k\ge 0$,
\[
\|\mathbb T_\varepsilon^k\|_{\mathcal V_w\to\mathcal V_w}\le M_w.
\]
(In many Markov settings one can take $M_w=1$.)
\item[(A3)] (\emph{Mixed closeness})
\[
\delta_\varepsilon:=\|\mathbb T_\varepsilon-\mathbb T_0\|_{\mathcal B_s\to\mathcal B_w}\to 0
\]
as $\varepsilon \to 0$.
\end{enumerate}

Then there exists a constant $K$, depending only on
$C$, $\rho$, and $M_w$, such that, for all sufficiently
small $\varepsilon$,

\begin{equation}\label{eq:resolvent_delta_log_delta}
\|(I-\mathbb T_\varepsilon)^{-1}-(I-\mathbb T_0)^{-1}\|_{\mathcal V_s\to\mathcal V_w}
\le
K\,\delta_\varepsilon\Bigl(1+\log\frac{1}{\delta_\varepsilon }\Bigr).
\end{equation}
In particular,
\[
\|(I-\mathbb T_\varepsilon)^{-1}-(I-\mathbb T_0)^{-1}\|_{\mathcal V_s\to\mathcal V_w}\xrightarrow[\varepsilon\to 0]{}0.
\]
\end{lemma}

\begin{proof}
By (A1), the Neumann series
\[
(I-\mathbb{T}_\varepsilon)^{-1}
=\sum_{k=0}^{\infty}\mathbb{T}_\varepsilon^k
\]
converges in operator norm on $\mathcal{V}_s$.
For every $k\geq 1$, the telescoping identity gives
\begin{equation}\label{eq:power_telescoping}
\mathbb{T}_\varepsilon^k-\mathbb{T}_0^k
=
\sum_{j=0}^{k-1}
\mathbb{T}_\varepsilon^{k-1-j}
(\mathbb{T}_\varepsilon-\mathbb{T}_0)
\mathbb{T}_0^j.
\end{equation}
Hence, by (A1)--(A3),
\[
\bigl\|\mathbb{T}_\varepsilon^k-\mathbb{T}_0^k
\bigr\|_{\mathcal{V}_s\to\mathcal{V}_w}
\leq
CM_w\delta_\varepsilon\sum_{j=0}^{k-1}\rho^j
\leq \frac{CM_w}{1-\rho}\delta_\varepsilon.
\]
On the other hand, the strong-to-weak embedding,
the triangle inequality, and (A1) give
\[
\begin{aligned}
\|\mathbb{T}_\varepsilon^k-\mathbb{T}_0^k
\|_{\mathcal{V}_s\to\mathcal{V}_w}
&\leq
\|\mathbb{T}_\varepsilon^k\|_{\mathcal{V}_s\to\mathcal{V}_s}
+
\|\mathbb{T}_0^k\|_{\mathcal{V}_s\to\mathcal{V}_s}
\\
&\leq 2C\rho^k.
\end{aligned}
\]
Thus we have two bounds for the difference of the powers:
one proportional to $\delta_\varepsilon$, and one
decaying exponentially in $k$.
For any $N\geq1$, the Neumann-series representations give
\[
\begin{aligned}
&(I-\mathbb{T}_\varepsilon)^{-1}
-(I-\mathbb{T}_0)^{-1}
\\
&\qquad=
\sum_{k=0}^{N-1}
(\mathbb{T}_\varepsilon^k-\mathbb{T}_0^k)
+
\sum_{k=N}^{\infty}
(\mathbb{T}_\varepsilon^k-\mathbb{T}_0^k).
\end{aligned}
\]
Using the first bound for the finite sum and the
exponential bound for the tail, we obtain
\[
\begin{aligned}
&\bigl\|
(I-\mathbb{T}_\varepsilon)^{-1}
-(I-\mathbb{T}_0)^{-1}
\bigr\|_{\mathcal{V}_s\to\mathcal{V}_w}
\\
&\qquad\leq
\frac{CM_w}{1-\rho}\delta_\varepsilon N
+2C\sum_{k=N}^{\infty}\rho^k
=
\frac{CM_w}{1-\rho}\delta_\varepsilon N
+\frac{2C}{1-\rho}\rho^N.
\end{aligned}
\]
For $0<\delta_\varepsilon<1$, choose
\(
N=\left\lceil
\frac{\log(1/\delta_\varepsilon)}{-\log\rho}
\right\rceil.
\)
Then $\rho^N\leq\delta_\varepsilon$ and
$N\leq 1+\log(1/\delta_\varepsilon)/(-\log\rho)$,
which proves \eqref{eq:resolvent_delta_log_delta},
with a constant depending only on $C$, $\rho$, and $M_w$.
If $\delta_\varepsilon=0$, the two operators agree
on $\mathcal{B}_s$, so their resolvents agree on
$\mathcal{V}_s$.
The asserted continuity follows.
\end{proof}

\subsection{Linear response and the series formula}

We are now ready to  derive the discrete-time linear response formula as a resolvent applied to a forcing term.
The argument is very similar  as the ones typically used in the autonomous case, but carried out for the fixed-point equation of the global transfer operator on the sequence space. This leads to obtaining rather compact and interpretable formulas.



\begin{theorem}[Linear response via the global-map resolvent]\label{thm:LR_strong}
Assume \eqref{eq:mass_preserving}, the weak power-boundedness assumption (A2) and Assumptions~\ref{ass:mu_bound}--\ref{ass:uniform_ELoM}.

 Define the difference quotients
\[
\boldsymbol h^\varepsilon := \frac{\boldsymbol\mu^\varepsilon-\boldsymbol\mu^0}{\varepsilon}\in\mathcal B_s,
\qquad \varepsilon \in [0,\varepsilon_0).
\]
Then:
\begin{enumerate}
\item The family $(\boldsymbol h^\varepsilon)_{\varepsilon}$ is bounded in $\mathcal B_s$ and admits a limit
\[
\boldsymbol\eta := \lim_{\varepsilon\to0}\boldsymbol h^\varepsilon
\quad\text{in }\mathcal B_w.
\]
\item The limit $\boldsymbol\eta\in\mathcal V_s$ is the unique solution in $\mathcal V_s$ of the resolvent equation
\begin{equation}\label{eq:LR_resolvent}
(I-\mathbb T)\boldsymbol\eta = \mathbb S^{-1}\boldsymbol g,
\qquad
\text{where}\ 
\mathbb T:=\mathbb T_0,\ \boldsymbol g=(g_n)_{n\in \mathbb Z} , \ g_n:=\dot L_n \mu_n.
\end{equation}
Equivalently,
\begin{equation}\label{eq:LR_neumann}
\boldsymbol\eta = (I-\mathbb T)^{-1}\mathbb S^{-1}\boldsymbol g
= \sum_{k=0}^{\infty}\mathbb T^k \mathbb S^{-1}\boldsymbol g
\quad\text{in }\mathcal B_s.
\end{equation}
\item (Coordinate-wise series formula.) For every $n\in\mathbb Z$,
\begin{equation}\label{eq:LR_series}
\eta_n = \sum_{k=0}^{\infty} L_{n-1}\cdots L_{n-k}\, \dot L_{n-k-1}\,\mu_{n-k-1},
\end{equation}
Here the product $L_{n-1}\cdots L_{n-k}$ is understood
as the identity when $k=0$, so that the first term
of the series is $\dot L_{n-1}\mu_{n-1}$.
The series converges absolutely in $B_s$, uniformly in $n$.
\end{enumerate}
\end{theorem}

\begin{proof}
From $\mathbb S\boldsymbol\mu^\varepsilon=\mathbb F_\varepsilon(\boldsymbol\mu^\varepsilon)$ and
$\mathbb S\boldsymbol\mu=\mathbb F(\boldsymbol\mu)$, subtract and divide by $\varepsilon$:
\begin{equation}\label{eq:quotient_identity}
\mathbb S\boldsymbol h^\varepsilon
=
\mathbb F_\varepsilon(\boldsymbol h^\varepsilon)
+
\boldsymbol g^\varepsilon,
\qquad  \boldsymbol g^\varepsilon:=(g_n^\varepsilon)_{n\in \mathbb Z}, \
(g_n^\varepsilon):=\frac{L_n^\varepsilon-L_n}{\varepsilon}\,\mu_n.
\end{equation}
Applying $\mathbb S^{-1}$ gives
\begin{equation}\label{eq:Res_eps_again}
(I-\mathbb T_\varepsilon)\boldsymbol h^\varepsilon = \mathbb S^{-1}\boldsymbol g^\varepsilon,
\qquad \mathbb T_\varepsilon=\mathbb S^{-1}\mathbb F_\varepsilon.
\end{equation}

Note that by Assumption~\ref{ass:strong_derivative},
\[
\|\boldsymbol g^\varepsilon-\boldsymbol g\|_{\mathcal B_s}
=
\sup_n\left\|
\frac{L_n^\varepsilon\mu_n-L_n \mu_n}{\varepsilon}-\dot L_n\mu_n
\right\|_s
\to 0
\]
where $g_n=\dot L_n\mu_n$ and $\boldsymbol g:=(g_n)_{n\in \mathbb Z}$.

Moreover, mass preservation (1) gives, for every $n$
and every admissible $\varepsilon>0$,
\[
g_n^\varepsilon(X)
=
\frac{(L_n^\varepsilon\mu_n)(X)-(L_n\mu_n)(X)}
     {\varepsilon}
=0.
\]
Thus $g^\varepsilon\in\mathcal V_s$.
Since $V_s$ is closed in $B_s$, the space
$\mathcal V_s=\ell^\infty(\mathbb Z;V_s)$ is closed
in $\mathcal B_s$.
The convergence $g^\varepsilon\to g$ in $\mathcal B_s$
therefore implies $g\in \mathcal V_s$.
In particular, $(\dot L_n\mu_n)(X)=0$ for every $n$.

\noindent \emph{Step 1: uniform invertibility on $\mathcal V_s$ and boundedness of $\boldsymbol h^\varepsilon$.}
By Lemma~\ref{lem:neumann}, for each $\varepsilon \in [0,\varepsilon_0)$ the inverse $(I-\mathbb T_\varepsilon)^{-1}$ exists on $\mathcal V_s$ and
\[
\|(I-\mathbb T_\varepsilon)^{-1}\|_{\mathcal V_s\to\mathcal V_s}\le \frac{C}{1-\rho}.
\]
Applying this to \eqref{eq:Res_eps_again} yields
\[
\|\boldsymbol h^\varepsilon\|_{\mathcal B_s}
\le \|(I-\mathbb T_\varepsilon)^{-1}\|\,\|\mathbb S^{-1}\boldsymbol g^\varepsilon\|_{\mathcal B_s}
\le \frac{C}{1-\rho}\,\|\boldsymbol g^\varepsilon\|_{\mathcal B_s}.
\]
Since $\boldsymbol g^\varepsilon\to \boldsymbol g$ in $\mathcal B_s$, the family $(\boldsymbol h^\varepsilon)$ is bounded in $\mathcal B_s$.

\noindent \emph{Step 2: identification of the limit and convergence of $\boldsymbol h^\varepsilon$.}

From \eqref{eq:Res_eps_again} we get 
\begin{equation}
(I-\mathbb T_\varepsilon)\boldsymbol h^\varepsilon = \mathbb S^{-1}(\boldsymbol g^\varepsilon- \boldsymbol g)+ \mathbb S^{-1} \boldsymbol g
\end{equation}

and

\begin{equation}
\boldsymbol h^\varepsilon = (I-\mathbb T_\varepsilon)^{-1} \mathbb S^{-1}(\boldsymbol g^\varepsilon- \boldsymbol g)+ (I-\mathbb T_\varepsilon)^{-1} \mathbb S^{-1} \boldsymbol g .
\end{equation}
Since $(I-\mathbb T_\varepsilon)^{-1}$ is uniformly bounded and $(\boldsymbol g^\varepsilon- \boldsymbol g)\to 0$ then $(I-\mathbb T_\varepsilon)^{-1} \mathbb S^{-1}(\boldsymbol g^\varepsilon- \boldsymbol g)\to 0$ in the strong norm.
By Lemma \ref{lem:resolvent_continuity_mixed} we also have that $ (I-\mathbb T_\varepsilon)^{-1} \mathbb S^{-1} \boldsymbol g \to  (I-\mathbb T_0)^{-1} \mathbb S^{-1} \boldsymbol g $ in the weak norm.
All together this implies $\boldsymbol h^{\varepsilon} \to \boldsymbol \eta$ in the weak norm as claimed.

\

\noindent \emph{Step 3: Neumann series and the coordinate-wise formula.}
Equation \eqref{eq:LR_resolvent} implies \eqref{eq:LR_neumann} by Lemma~\ref{lem:neumann} at $\varepsilon=0$.
Taking the $n$-th coordinate, note that $(\mathbb S^{-1}\boldsymbol g)_n=g_{n-1}=\dot L_{n-1}\mu_{n-1}$ and
\[
(\mathbb T^k \boldsymbol v)_n = L_{n-1}\cdots L_{n-k}\, v_{n-k}.
\]
Therefore
\[
\eta_n = \sum_{k=0}^{\infty} L_{n-1}\cdots L_{n-k}\, (\mathbb S^{-1}\boldsymbol g)_{n-k}
= \sum_{k=0}^{\infty} L_{n-1}\cdots L_{n-k}\, \dot L_{n-k-1}\mu_{n-k-1},
\]
which is \eqref{eq:LR_series}.

Finally, absolute convergence in $B_s$ follows from Assumptions~\ref{ass:uniform_ELoM}, \ref{ass:mu_bound}, and \eqref{eq:dotL_bound}:
each summand has zero mass and
\[
\|L_{n-1}\cdots L_{n-k}\dot L_{n-k-1}\mu_{n-k-1}\|_s
\le
C\rho^k \Big(\sup_j\|\dot L_j\mu_{j}\|_{s}\Big)
\]
so, the series is dominated by a geometric series uniformly in $n$.
\end{proof}

\section{Linear response for sequential composition of expanding maps}\label{sequential}

In this section we apply Theorem \ref{thm:LR_strong} to establish linear response for sequential composition of deterministic expanding maps.
The kind systems we will consider  are sequential composition of maps that are close to a certain expanding map $T_0$ (the admitted distance will be estimated by the properties of $T_0$), and we will consider additive deterministic perturbations of the maps.
In the following subsections we will see how to verify the assumptions needed to apply Theorem \ref{thm:LR_strong}, starting from the exponential loss of memory, { which will be provided by the results of Appendix \ref{appendixa}. The results of
Appendix~\ref{appendixa} would also allow one to treat similar examples in
which the expanding maps vary slowly along the sequence (see  Theorem~\ref{thm:slow-variation-loss-memory}). To avoid a repetition of nearly identical arguments, we
do not spell out this extension in detail.}

\subsection{Uniform exponential loss of memory for sequential composition of expanding maps near $T_0$}\label{5.1}

\noindent {\bf General setting and assumptions. } In this section we set and work with
\[
B_s = W^{1,1}(\mathbb S^1),\qquad B_w = L^1(\mathbb S^1).
\]

Let $T_0:\mathbb S^1\to\mathbb S^1$ be a $C^4$ expanding map of degree $n\ge2$ such that
\[
\lambda_0:=\inf_{\mathbb S^1}|T_0'|>1,\qquad 
M_0:=\|T_0'\|_\infty<\infty,\qquad 
M_2:=\|T_0''\|_\infty<\infty,
\]
and assume that $T_0$ is mixing. 
We recall that under these assumptions the transfer  operator $L_0$ associated to $T_0$ has a spectral gap on $W^{1,1}$
(equivalently, exponential contraction on $V_s:=\{f\in W^{1,1}:\int f\,dm=0\}$).

Choose an auxiliary radius $\bar\delta$ such that
\[
0<\bar\delta<\lambda_0-1,
\]
and set
\[
\lambda_1:=\frac{1}{\lambda_0-\bar\delta}\in(0,1),
\qquad
B:=1+\frac{M_2+\bar\delta}
              {(\lambda_0-\bar\delta)^2}.
\]
These constants give a common one-step Lasota--Yorke
inequality for the transfer operators associated with
$C^3$ maps $T$ satisfying
$\|T-T_0\|_{C^2}\leq\bar\delta$:
\begin{equation}
\|L_Tf\|_{L^1}\leq\|f\|_{L^1},
\qquad
\|L_Tf\|_{W^{1,1}}
\leq\lambda_1\|f\|_{W^{1,1}}+B\|f\|_{L^1},
\qquad f\in W^{1,1}(S^1).
\end{equation}
In particular, these inequalities hold for $L_0$ (see for example  \cite[Lemma~18]{FG2025}).

In particular in this case, the assumption(A2) is satisfied with $M_w=1$.

Let $M\in\mathbb N$ be such that
\begin{equation}\label{eq:M_choice_ML2}
\lambda_1^{M}\le \frac{1}{10\left(\frac{B}{1-\lambda_1}+1\right)}
\quad\text{and}\quad
\|L_0^{M}v\|_{L^1}\le \frac{1-\lambda_1}{10B}\|v\|_{W^{1,1}}\qquad \forall v\in V_s.
\end{equation}
(Existence of such an $M$ follows for example from \cite[Proposition~19]{FG2025})

Let \begin{equation}\label{eq:CT0_explicit_final}
C(T_0)=2
n\bigl(\|T_0'\|_\infty+\lambda_0-1\bigr)\left(\frac{1}{\lambda_0^2}+\frac{M_2}{\lambda_0^3}+\frac{1}{\lambda_0}\right)
\end{equation}
 be the mixed-norm constant from Proposition~\ref{prop:mixed_continuity_covering_self} in Appendix B, namely for any $T$ such that $\|T_0-T\|_{C^2} \leq \lambda_0 -1$

\begin{equation}\label{eq:mixed_bound_CT0}
\|L_T-L_0\|_{W^{1,1}\to L^1}\le C(T_0)\,\|T-T_0\|_{C^2}.
\end{equation}

\begin{lemma}[Unif. exponential loss of memory for seq. expanding maps.]
\label{lem:ELoM_near_T0}
Fix a number $\delta_*>0$ such that
\begin{equation}\label{eq:delta_star_basic}
0<\delta_*\leq\bar\delta.
\end{equation}

Assume in addition that $\delta_\ast$ is small enough so that
\begin{equation}\label{eq:delta_star_ML3_revised}
C(T_0)\,\delta_\ast
\ \le\
\frac{7(1-\lambda_1)^2}{10MB\left(\frac{1}{1-\lambda_1}+B\right)}.
\end{equation}

Define the class of maps
\begin{equation}\label{eq:ST0_def_revised}
\mathcal S_{T_0}(\delta_\ast):=
\Bigl\{\,T:\mathbb S^1\to\mathbb S^1\ \text{$C^3$ uniformly expanding of degree $n$}\ :\ \|T-T_0\|_{C^2}< \delta_\ast \Bigr\}.
\end{equation}

Let $(T_i)_{i\ge0}$ be any sequence with $T_i\in\mathcal S_{T_0}(\delta_\ast)$ for all $i$, and let $L_i:=L_{T_i}$.
There exist constants $C_{\mathrm{ELoM}}\ge1$ and $\lambda>0$ such that for all $j,n\in\mathbb N$ and all $g\in V_s$,
\begin{equation}\label{eq:ELoM_conclusion_revised}
\|L^{(j,j+n-1)}g\|_{W^{1,1}}\le C_{\mathrm{ELoM}}e^{-\lambda n}\|g\|_{W^{1,1}}.
\end{equation}
\end{lemma}

\begin{proof}
We verify the assumptions of Lemma~\ref{losmem}.
Since $\delta_*\leq\bar\delta$, the common Lasota--Yorke
inequality established above gives (ML1).
The choice of $M$ in \eqref{eq:M_choice_ML2} gives (ML2).
Finally, \eqref{eq:mixed_bound_CT0} implies
\[
\|L_i-L_0\|_{W^{1,1}\to L^1}
\leq C(T_0)\|T_i-T_0\|_{C^2}
\leq C(T_0)\delta_*.
\]
Together with \eqref{eq:delta_star_ML3_revised},
this verifies (ML3).
Lemma~\ref{losmem} therefore yields the claimed exponential
loss of memory, uniformly in the initial time.
\end{proof}

\subsection{Perturbations of the sequential composition and verification of the Assumptions \ref{ass:mu_bound}, \ref{ass:strong_derivative}. }
\label{sec:verify_expanding}
In this section we define the kinds of perturbations we mean to apply to a sequential composition of expanding maps and verify that Assumptions \ref{ass:mu_bound}, \ref{ass:strong_derivative}  hold.
The kind of perturbations we consider are given by the application of a further diffeomorphism after the application of the dynamics, adding in some sense a further "small kick" to the result of the unperturbed dynamics.
This will lead to the definition of a certain kind of family of sequential maps  $T_n^\varepsilon$ parametrized by  $\varepsilon \in [0,\varepsilon_0 )$. 

Throughout this section, we continue working with the general setting and notations stated at beginning of Section \ref{5.1}.
We suppose $T_0$, $\delta_\ast$ and the set $\mathcal S_{T_0}(\delta_\ast)$ being fixed.

\subsubsection{Post-composition perturbations.}\label{5-2-1}

Let $(T_n)_{n\in\mathbb Z}$ be a sequence of $C^4$
expanding covering maps such that
\[
\sup_{n\in\mathbb Z}\|T_n\|_{C^4}<\infty,
\qquad
\sup_{n\in\mathbb Z}\|T_n-T_0\|_{C^2}
\leq\delta'<\delta_*.
\]

Let $X\in C^3(\mathbb S^1)$ be a smooth vector field on the circle, and let $(h_\varepsilon)_{\varepsilon \in [0, \varepsilon_0)}$ be a $C^3$ family of $C^3$ diffeomorphisms of $\mathbb S^1$ with $h_0=\mathrm{id}$ and
\begin{equation}\label{eq:kick_family}
h_\varepsilon(x)=x+\varepsilon X(x)+r^\varepsilon(x),
\qquad  
\frac{\|r^\varepsilon\|_{C^3}}{|\varepsilon|}\xrightarrow[\varepsilon\to0]{}0.
\end{equation}
Define the perturbed sequential maps by \emph{post-composition}
\begin{equation}\label{eq:T_kick}
T_n^\varepsilon := h_\varepsilon\circ T_n.
\end{equation}

The uniform $C^2$ bound on $(T_n)$ and the chain rule
give a constant $C<\infty$, independent of $n$ and
$\varepsilon$, such that
\[
\sup_{n\in\mathbb Z}
\|T_n^\varepsilon-T_n\|_{C^2}
\leq C\|h_\varepsilon-\operatorname{id}\|_{C^2}
=O(\varepsilon).
\]
Consequently, after reducing $\varepsilon_0$ if necessary,
\[
\sup_{n\in\mathbb Z}
\|T_n^\varepsilon-T_0\|_{C^2}
\leq
\delta'
+C\|h_\varepsilon-\operatorname{id}\|_{C^2}
<\delta_*
\qquad
\text{for all }\varepsilon\in[0,\varepsilon_0).
\]
Thus $T_n^\varepsilon\in\mathcal S_{T_0}(\delta_*)$
for every $n$ and every admissible $\varepsilon$.
In particular, Lemma~12 gives exponential loss of memory
with constants independent of $\varepsilon$ and of the
initial time.

Let $L_n$ and $L_n^\varepsilon$ be the Perron--Frobenius operators associated to $T_n$ and $T_n^\varepsilon$, respectively.  $L_n^\varepsilon = L_{h_\varepsilon}\, L_n$ where $L_{h_\varepsilon}$ is the operator associated to $h_\varepsilon$.

\subsubsection{Uniform bounds for the equivariant family (Assumption~\ref{ass:mu_bound})}\label{5.2.2}
Under the following assumption on the maps $T_n^\varepsilon$  we can establish Assumption~\ref{ass:mu_bound} for the associated system of transfer operators.

The following global regularity property  for the maps $T_n^\varepsilon$ directly follows by the way the perturbations are constructed. 
\begin{lemma}\label{N}
Under the above assumptions for the maps $T_n^\varepsilon$, there exists  $N>0$ such that $\forall n\in \mathbb Z,\varepsilon\in [0,\varepsilon_0)$ one has
$||T_n^\varepsilon||_{C^3}\leq N.$
\end{lemma}

This allows to prove the following uniform bound on the regularity of the equivariant measures $\mu^\varepsilon_n$.
\begin{lemma}\label{15}
Let  $T_0$ be a uniformly expanding map on $S^1$, let  $\delta_\ast$ and  $\mathcal S_{T_0}(\delta_\ast)$ be as in Section \ref{5.1}.
Let $T_n^\varepsilon \in S_{T_0}(\delta_\ast)$ with  $\varepsilon\in [0,\varepsilon_0)$  be a family of expanding maps coming from the deterministic perturbation of a family of maps belonging to $S_{T_0}(\delta_\ast)$  as described in in Section \ref{5-2-1}.

Then each sequence  $(T_n^\varepsilon)_{n\in \mathbb Z}$ has an equivariant family
$\boldsymbol\mu^\varepsilon=(\mu_n^\varepsilon)_{n\in\mathbb Z}\in \ell^\infty(\mathbb Z, W^{2,1})$ such that
\[
\mu_{n+1}^\varepsilon=L_n^\varepsilon\mu_n^\varepsilon,\qquad \mu_n^\varepsilon\ge 0,\qquad \int \mu_n^\varepsilon dm = 1,
\]
and
\[
\sup_{0\leq \varepsilon <\varepsilon_0}\sup_{n\in\mathbb Z}\|\mu_n^\varepsilon\|_{W^{2,1}} < \infty.
\]
Furthermore 
\[
\sup_{n\in\mathbb Z}\|\mu^0_n\|_{W^{3,1}} < \infty.
\]
In particular, Assumption~\ref{ass:mu_bound} holds for this family.
\end{lemma}
 \begin{proof}
 By Lemma \ref{N} the maps $T_n^\varepsilon$ have uniform $C^3$ regularity, furthermore, if $\varepsilon_0$ is small enough they are also uniformly expanding.  It is well known (see e.g. \cite{FG2025}, Lemma~18) that such a set of maps   satisfy uniformly Lasota Yorke  inequalities on $W^{i,1}$,
in particular for $i=1,2$:
\[
\|L_n^\varepsilon f\|_{W^{1,1}}\le a\|f\|_{W^{1,1}}+b\|f\|_{L^1},\qquad
\|L_n^\varepsilon f\|_{W^{2,1}}\le a_2\|f\|_{W^{2,1}}+b_2\|f\|_{W^{1,1}},
\]
with constants independent of $n$ and $\varepsilon$.
Fix a reference density $f_\ast\in W^{3,1}$ with $\int f_\ast\,dm=1$, and define pullback iterates
$\mu_{n}^{\varepsilon,(k)}:=L_{n-1}^\varepsilon\cdots L_{n-k}^\varepsilon f_\ast$.
The Lasota-Yorke bounds give uniform $W^{2,1}$ bounds for $\mu_{n}^{\varepsilon,(k)}$ and uniform $W^{3,1}$ bounds for $\mu_{n}^{0,(k)}$ since the associated transfer operators are related respectively to the $C^3$ expanding maps $T^\varepsilon_n$ and to the $C^4$ expanding maps $T^0_n$.

We now justify convergence in the stronger Sobolev norms.
Write
\[
\mathcal L_{j,k}^\varepsilon
:=L_{j+k-1}^\varepsilon\cdots L_j^\varepsilon,
\qquad
\mathcal L_{j,0}^\varepsilon:=I.
\]
The uniform expansion and regularity assumptions give
constants $a_r\in(0,1)$ and $b_r<\infty$ such that
\[
\|L_n^\varepsilon f\|_{W^{r,1}}
\leq a_r\|f\|_{W^{r,1}}
     +b_r\|f\|_{W^{r-1,1}},
\]
uniformly in $n$, for $r=2$ and all admissible
$\varepsilon$, and for $r=3$ when $\varepsilon=0$.

These inequalities upgrade the exponential loss of
memory of Lemma~12 to the corresponding stronger norms.
Indeed, suppose that, for zero-mass $v$,
\[
\|\mathcal L_{j,k}^\varepsilon v\|_{W^{r-1,1}}
\leq C_{r-1}\rho_{r-1}^k\|v\|_{W^{r-1,1}},
\qquad \rho_{r-1}<1.
\]
Iterating the Lasota--Yorke inequality yields
\[
\begin{aligned}
\|\mathcal L_{j,k}^\varepsilon v\|_{W^{r,1}}
&\leq a_r^k\|v\|_{W^{r,1}}
\\
&\quad+
b_r C_{r-1}
\sum_{\ell=0}^{k-1}
a_r^{k-1-\ell}\rho_{r-1}^{\ell}
\|v\|_{W^{r-1,1}}.
\end{aligned}
\]
Choosing $\rho_r\in(\max\{a_r,\rho_{r-1}\},1)$
and estimating the geometric convolution gives
\[
\|\mathcal L_{j,k}^\varepsilon v\|_{W^{r,1}}
\leq C_r\rho_r^k\|v\|_{W^{r,1}}.
\]
Starting from Lemma~12, this proves the estimate for
$r=2$, uniformly in $\varepsilon$, and then for
$r=3$ at $\varepsilon=0$.

For the pullback iterates defined above, we have
\[
\mu_n^{\varepsilon,(k+m)}
-\mu_n^{\varepsilon,(k)}
=
\mathcal L_{n-k,k}^\varepsilon
\left(
\mathcal L_{n-k-m,m}^\varepsilon f_*-f_*
\right).
\]
The expression in parentheses has zero mass and is
uniformly bounded in $W^{2,1}$, and also in $W^{3,1}$
when $\varepsilon=0$, by the bounds already established.
Consequently,
\[
\|\mu_n^{\varepsilon,(k+m)}
-\mu_n^{\varepsilon,(k)}\|_{W^{r,1}}
\leq K_r\rho_r^k,
\]
with constants independent of $n,k,m$ and, for $r=2$,
of $\varepsilon$.

Thus the pullback iterates converge in $W^{2,1}$,
and in $W^{3,1}$ for the unperturbed system.
Their limits inherit the uniform bounds stated above.
Positivity and unit mass pass to the limit by
$L^1$ convergence, while boundedness of the transfer
operators implies equivariance.
Finally, Proposition
\ref{prop:unique_fixedpoint} gives uniqueness among
equivariant families of probability densities
uniformly bounded in $W^{1,1}$.
\end{proof}

{
\begin{remark} [Dynamical significance of the equivariant family]
The uniqueness established above holds among equivariant families whose densities are uniformly bounded in \(W^{1,1}\). This is compatible with the existence of uncountably many absolutely continuous equivariant families under weaker assumptions, as shown in Proposition 24 of \cite{DragicevicGonzalezTokmanSedro2024}.

The family considered here is dynamically distinguished by its pullback attraction property. Indeed, writing \(h_n\) for the density of \(\mu_n\) and \(M=\sup_j\|h_j\|_{W^{1,1}}\), exponential loss of memory gives, for every probability density \(f\in W^{1,1}\),
\[
\|L_{n-1}\cdots L_{n-k}f-h_n\|_{W^{1,1}}
\le C\rho^k\bigl(\|f\|_{W^{1,1}}+M\bigr).
\]
Thus \(h_n\) describes the statistical state at time \(n\) obtained by initializing the system in the remote past, independently of the initial density. This convergence also holds uniformly over initial densities in any bounded subset of \(W^{1,1}\).

Consequently, the family is selected by the dynamics through loss of memory of the initial preparation, which motivates studying its linear response. This interpretation parallels the characterization of the distinguished sequential family as an SRB state through pullback limits in Section 4.1 of the same reference.
\end{remark}
}

We remark that the regularity established above for the family $\mu_{n}^{\varepsilon}$ is more than necessary to verify  Assumption~\ref{ass:mu_bound}, taking $W^{1,1}$ as a strong space. However this will be useful to verfy Assumption \ref{ass:strong_derivative}, as we will see in the next section.

\subsubsection{Strong differentiability (Assumption~\ref{ass:strong_derivative})}
\label{subsubsec:strongdiff_kick}

In this subsection we verify Assumption~\ref{ass:strong_derivative} for the above family of perturbed sequential  transfer operators $L_n^\varepsilon$.

\paragraph{Derivative of the diffeomorphism transfer operator.}
For a $C^2$ orientation preserving diffeomorphism $h$ of $\mathbb S^1$, the Perron--Frobenius operator is
\[
(L_h u)(x)=\frac{u(h^{-1}(x))}{h'(h^{-1}(x))}.
\]

We recall the derivative formula for the transfer operator
of a diffeomorphism close to the identity, emphasizing
uniform convergence on bounded subsets of $W^{3,1}$.
This uniformity is needed below to differentiate along
the time-dependent equivariant family.

\begin{lemma}[Derivative operator]\label{lem:Dh}
Assume \eqref{eq:kick_family}. Define the first-order operator
\footnote{Here \(X\) is a \(C^3\) vector field on \(\mathbb S^1\), identified
in the standard coordinate with a \(1\)-periodic scalar function. Thus \(Xu\)
denotes pointwise multiplication.}
\begin{equation}\label{eq:D_operator}
Du:=-(Xu)',
\qquad
u\in W^{2,1}(\mathbb S^1).
\end{equation}
Then:
\begin{enumerate}
\item[{(i)}]
The operator
\begin{equation*}
D:W^{2,1}(\mathbb S^1)\longrightarrow W^{1,1}(\mathbb S^1)
\end{equation*}
is bounded. More precisely,
\begin{equation*}
\|Du\|_{W^{1,1}}
\leq
C_D\|u\|_{W^{2,1}},
\qquad
C_D:=3\|X\|_{C^2}.
\end{equation*}

\item[{(ii)}]
For every \(u\in W^{2,1}(\mathbb S^1)\),
\begin{equation}\label{eq:Dh_diffquot}
\lim_{\varepsilon\to0}
\left\|
\frac{L_{h_\varepsilon}u-u}{\varepsilon}-Du
\right\|_{W^{1,1}}
=0.
\end{equation}

Moreover, the convergence is uniform for \(u\) in bounded subsets of
\(W^{3,1}(\mathbb S^1)\). More precisely,
\begin{equation}\label{eq:Dh_uniform_convergence}
\sup_{\substack{u\in W^{3,1}(\mathbb S^1)\\
\|u\|_{W^{3,1}}\leq 1}}
\left\|
\frac{L_{h_\varepsilon}u-u}{\varepsilon}-Du
\right\|_{W^{1,1}}
\longrightarrow 0
\qquad
\text{as }\varepsilon\to0.
\end{equation}

\end{enumerate}
\end{lemma}

\begin{proof}
The product rule gives
\[
\|Du\|_{W^{1,1}}
=\|(Xu)'\|_{W^{1,1}}
\leq 3\|X\|_{C^2}\|u\|_{W^{2,1}},
\]
proving (i).

We first record two estimates valid uniformly for
$C^3$ diffeomorphisms $h$ sufficiently close to the identity:
\[
\begin{aligned}
\|L_hf\|_{W^{2,1}}
&\leq C\|f\|_{W^{2,1}},\\
\|(L_h-I)f\|_{W^{2,1}}
&\leq C\|h-\mathrm{id}\|_{C^3}\|f\|_{W^{3,1}}.
\end{aligned}
\]
Indeed, writing $g=h^{-1}$ and $L_hf=(f\circ g)g'$,
these follow by differentiating twice, changing variables,
and using
\[
\|v\circ g-v\|_{L^1}
\leq C\|g-\mathrm{id}\|_\infty\|v'\|_{L^1},
\qquad v\in W^{1,1}.
\]
The latter estimate follows from the fundamental theorem
of calculus along $\mathrm{id}+t(g-\mathrm{id})$.
Here and below, constants are independent of sufficiently
small $\varepsilon$ and of the functions being estimated.

Set
\[
v_\varepsilon:=h_\varepsilon-\mathrm{id}
=\varepsilon X+r_\varepsilon,
\qquad
h_{\varepsilon,t}:=\mathrm{id}+t v_\varepsilon,
\quad 0\leq t\leq1.
\]
These are diffeomorphisms uniformly close to the identity
in $C^3$. Differentiating
$\int \varphi(h_{\varepsilon,t}(x))u(x)\,dx$
with respect to $t$, for smooth test functions $\varphi$,
and integrating in $t$ gives the exact identity
\[
L_{h_\varepsilon}u-u
=
-\partial_x\int_0^1
L_{h_{\varepsilon,t}}(v_\varepsilon u)\,dt,
\]
which holds in $W^{1,1}$ for $u\in W^{2,1}$.
Since $Du=-\partial_x(Xu)$, it follows that
\[
\frac{L_{h_\varepsilon}u-u}{\varepsilon}-Du
=
-\partial_x\int_0^1
\left[
(L_{h_{\varepsilon,t}}-I)(Xu)
+
L_{h_{\varepsilon,t}}
\left(\frac{r_\varepsilon}{\varepsilon}u\right)
\right]dt.
\]

For $u\in W^{3,1}$, the two estimates above and the
product rule therefore yield
\[
\left\|
\frac{L_{h_\varepsilon}u-u}{\varepsilon}-Du
\right\|_{W^{1,1}}
\leq
C\left(
\|v_\varepsilon\|_{C^3}
+\frac{\|r_\varepsilon\|_{C^2}}{|\varepsilon|}
\right)\|u\|_{W^{3,1}}.
\]
The coefficient on the right tends to zero by
\eqref{eq:kick_family}. This proves the uniform convergence
\eqref{eq:Dh_uniform_convergence}, which is the estimate
needed for the sequential application.

Finally, the same integral identity and the uniform
$W^{2,1}$ bound for $L_{h_{\varepsilon,t}}$ give
\[
\left\|
\frac{L_{h_\varepsilon}u-u}{\varepsilon}
\right\|_{W^{1,1}}
\leq
C\left\|\frac{v_\varepsilon}{\varepsilon}\right\|_{C^2}
\|u\|_{W^{2,1}}
\leq C'\|u\|_{W^{2,1}}.
\]
Thus the difference quotients are uniformly bounded
from $W^{2,1}$ to $W^{1,1}$.
Since $D$ is bounded between the same spaces and
$W^{3,1}$ is dense in $W^{2,1}$, convergence extends
to every fixed $u\in W^{2,1}$, proving
\eqref{eq:Dh_diffquot}.
\end{proof}

\paragraph{Verification of Assumption~\ref{ass:strong_derivative}.}
Let $\boldsymbol\mu=(\mu_n)_{n\in\mathbb Z}$ be the unperturbed equivariant family for $(L_n)$.
Define the forcing sequence
\begin{equation}\label{eq:forcing_kick}
g_n  := D(L_n(\mu_n))\in W^{1,1}.
\end{equation}

\begin{lemma}\label{lem:strong_derivative_kick}
Under the assumptions of Section~\ref{5-2-1},
Assumption~\ref{ass:strong_derivative} holds with
$B_s=W^{1,1}$, $B_w=L^1$, and forcing
$g_n=D(L_n\mu_n)=D\mu_{n+1}$.
\end{lemma}

\begin{proof}
Since $L_n^\varepsilon=L_{h_\varepsilon}L_n$
and $L_n\mu_n=\mu_{n+1}$,
\[
\frac{L_n^\varepsilon\mu_n-L_n\mu_n}{\varepsilon}
=
\frac{L_{h_\varepsilon}\mu_{n+1}-\mu_{n+1}}{\varepsilon}.
\]
By Lemma~\ref{15}, the densities $\mu_n$ are uniformly
bounded in $W^{3,1}$. The uniform convergence in
Lemma~\ref{lem:Dh} therefore implies that these quotients
converge to $D\mu_{n+1}$ in $W^{1,1}$, uniformly in $n$.
The bound $\sup_n\|D\mu_{n+1}\|_{W^{1,1}}<\infty$
follows from the boundedness of $D:W^{2,1}\to W^{1,1}$.

Finally, Proposition~\ref{prop:mixed_continuity_covering_self}
and the uniform expansion and $C^3$ bounds give
\[
\sup_n\|L_n^\varepsilon-L_n\|_{W^{1,1}\to L^1}
\leq C\sup_n\|T_n^\varepsilon-T_n\|_{C^2}
=O(\varepsilon),
\]
where the last estimate follows from
Section~\ref{5-2-1}.
This verifies the remaining mixed-continuity condition.
\end{proof}

\subsection{Conclusion: linear response for sequential expanding maps by deterministic perturbations}

\begin{theorem}[Linear response for sequential expanding maps]\label{thm:LR_expanding_kick_final}
Fix a $C^4$ expanding map $T_0:S^1\to S^1$ of degree
$d\geq 2$, and let $\delta_*>0$ be as in Lemma~12.
Let $(T_n)_{n\in\mathbb Z}$ be a sequence of $C^4$
expanding maps of degree $d$ such that
\[
\sup_{n\in\mathbb Z}\|T_n\|_{C^4}<\infty,
\qquad
\sup_{n\in\mathbb Z}\|T_n-T_0\|_{C^2}
\leq \delta'<\delta_*.
\]
Let $(h_\varepsilon)_{\varepsilon\in[0,\varepsilon_0)}$
be a $C^3$ family of $C^3$ diffeomorphisms of $S^1$,
with $h_0=\operatorname{id}$ and
\[
h_\varepsilon(x)
=x+\varepsilon X(x)+r^\varepsilon(x),
\qquad
X\in C^3(S^1),
\qquad
\frac{\|r^\varepsilon\|_{C^3}}{\varepsilon}
\longrightarrow 0
\quad\text{as }\varepsilon\downarrow0.
\]
Define
\[
T_n^\varepsilon:=h_\varepsilon\circ T_n,
\]
and let $L_n^\varepsilon$ be the associated
Perron--Frobenius operators.

Then, after reducing $\varepsilon_0$ if necessary,
the hypotheses of Theorem~11 hold with
\[
B_s=W^{1,1}(S^1),
\qquad
B_w=L^1(S^1).
\]
In particular, let
$\boldsymbol{\mu}^\varepsilon
=(\mu_n^\varepsilon)_{n\in\mathbb Z}$
denote the unique equivariant family of $(T^\varepsilon_n)$ made of probability
densities uniformly bounded in $W^{1,1}$.
The limit
\[
\boldsymbol{\eta}
=
\lim_{\varepsilon\downarrow0}
\frac{\boldsymbol{\mu}^\varepsilon
      -\boldsymbol{\mu}^0}{\varepsilon}
\]
exists in $\ell^\infty(\mathbb Z;L^1(S^1))$.
Moreover,
$\boldsymbol{\eta}\in\ell^\infty(\mathbb Z;V_s)$
and it satisfies \eqref{eq:LR_neumann} and \eqref{eq:LR_series}. 
\end{theorem}

\begin{proof}
The exponential loss of memory, Assumption \ref{ass:uniform_ELoM},
is verified under the hypotheses of the statement in Section \ref{5.1}.
The uniform strong bounds required in Assumption \ref{ass:mu_bound} was verified in Section \ref{5.2.2}.
The differentiability, Assumption \ref{ass:strong_derivative} was verified in Section \ref{subsubsec:strongdiff_kick}.
The application of Theorem~\ref{thm:LR_strong} directly lead to the statement.
\end{proof}

\begin{remark}[Extension to slowly varying expanding maps]
The proximity assumption in the preceding theorem can be
replaced by the slow-variation hypotheses of
Appendix~\ref{appendixa}. Retain uniform expansion,
the uniform $C^4$ bound on the unperturbed maps, and the
assumptions on post-composition perturbations.
The same linear-response conclusion holds if the associated
transfer operators satisfy $(\mathrm{ML}1)$,
$(\mathrm{ML}2')$, and $(\mathrm{ML}3')$, with constants
uniform in time and for sufficiently small perturbation
parameters.

Slow variation is a sufficient condition for memory loss
within our general strong--weak Lasota--Yorke framework,
which applies beyond expanding maps; it is not required
by the abstract response theorem itself.
For expanding maps, alternative approaches, including
the coupling method of Ott, Stenlund and Young
\cite{OttStenlundYoung2009} and the Lasota--Yorke and
sequential covering framework of Dolgopyat and Hafouta
\cite{DolgopyatHafouta2025}, establish memory loss under
different assumptions, allowing sequences without
slow temporal variation.
\end{remark}

{
\section{Response for sequential random systems with additive noise: Application to the Euler-Maruyama Integration Scheme}
\label{sec:LR_random_noise}
}
In this section we apply the abstract linear response theorem
(Theorem~\ref{thm:LR_strong}) to sequential random maps with additive
noise on a compact phase space. 
We will see a general setting under which we can establish a linear response for such systems (see Theorem \ref{thm:LR_noise}).
To show the flexibility of this approach then we will apply it to a non-trivial, physically interesting system such as a suitable discretization of the Ghil-Sellers energy balance model  (see Section \ref{subsec:LR_Ghil_Sellers}).

We
consider either
\[
X=\mathbb T^d=(\mathbb S^1)^d
\qquad\text{or}\qquad
X=I^d,
\]
where \(I=[a,b]\subset\mathbb R\) is a compact interval. In the torus
case, addition is understood modulo \(\mathbb Z^d\). In the cube case,
the additive perturbation is reflected at the boundary by means of a
standard folding map defined below. We denote by \(m\) the normalized
Lebesgue measure on \(X\).

The assumptions below are formulated directly in terms of the density
of the additive noise. The only regularity condition needed for the
linear-response argument is \(W^{1,1}\)-regularity of this density.
In particular, no differentiability of the reflection map is required.

\subsection{Periodic and reflected additive noise}
We introduce two examples of random systems with additive noise on a compact domain, corresponding to periodic or reflecting boundary conditions.
\paragraph{Periodic additive noise on \(\mathbb T^d\).}
Let \((\xi_n)_{n\in\mathbb Z}\) be i.i.d.\ random variables on
\(\mathbb T^d\) with common density
\[
q\in W^{1,1}(\mathbb T^d),
\qquad
q\geq0,
\qquad
\int_{\mathbb T^d}q\,dm=1.
\]
We assume that the density is uniformly positive:
\begin{equation}
\label{eq:noise_positive_torus}
q(y)\geq\alpha>0
\qquad\text{for \(m\)-a.e.\ }y\in\mathbb T^d.
\end{equation}
If \(z\in\mathbb T^d\), the random variable \(z+\xi_n\) has density
\begin{equation}
\label{eq:torus_noise_kernel}
k_{\mathbb T}(z,y):=q(y-z)
\end{equation}
with respect to \(m\).

\paragraph{Reflected additive noise on \(I^d\).}
Set \(\ell=b-a\). The one-dimensional folding map
\(\mathcal R_I:\mathbb R\to I\) is defined by
\begin{equation}
\label{eq:folding_map_1d}
\mathcal R_I(s)
:=
a+\ell-
\left|
\bigl((s-a)\!\!\!\pmod{2\ell}\bigr)-\ell
\right|.
\end{equation}
It repeatedly reflects the real line at the endpoints \(a\) and \(b\).
The reflection map on the cube is defined componentwise:
\[
\mathcal R(x_1,\ldots,x_d)
:=
\bigl(\mathcal R_I(x_1),\ldots,\mathcal R_I(x_d)\bigr).
\]

Let \((\xi_n)_{n\in\mathbb Z}\) be i.i.d.\ random variables in
\(\mathbb R^d\) with common density
\[
q\in W^{1,1}(\mathbb R^d),
\qquad
q\geq0,
\qquad
\int_{\mathbb R^d}q(x)\,dx=1.
\]
Given \(z\in\mathbb R^d\), the reflected additive-noise step is
\[
Y=\mathcal R(z+\xi_n).
\]
Its law is the pushforward under \(\mathcal R\) of the probability
measure with density \(q(\cdot-z)\). We denote its density with
respect to \(m\) by \(k_I(z,\cdot)\), so that
\begin{equation}
\label{eq:reflected_kernel_pushforward}
k_I(z,\cdot)\,m
=
\mathcal R_*\bigl(q(\cdot-z)\,dx\bigr).
\end{equation}

\subsection{Regularity of the noise kernels}

We now study some regularity property of such systems. This will be used later in the main
linear-response argument.

\begin{lemma}[Differentiability of translated and reflected noise kernels]
\label{lem:noise_kernel_differentiability}
Assume that \(q\in W^{1,1}\) on \(\mathbb T^d\) in the periodic case,
or on \(\mathbb R^d\) in the reflected case.

\begin{enumerate}
\item In the periodic case, the map
\[
z\longmapsto k_{\mathbb T}(z,\cdot)
\]
is continuously differentiable from \(\mathbb T^d\) to
\(L^1(\mathbb T^d,m)\), and
\begin{equation}
\label{eq:torus_noise_derivative}
D_1k_{\mathbb T}(z,\cdot)[u]
=
-\nabla q(\cdot-z)\cdot u
\end{equation}

where $D_1$ denotes the Fréchet derivative with respect to the first variable, viewing $z\mapsto k_X(z,\cdot)$ as an $L^1(X,m)$-valued map.

\item In the reflected case, the map
\[
z\longmapsto k_I(z,\cdot)
\]
is continuously differentiable from \(\mathbb R^d\) to
\(L^1(I^d,m)\), and
\begin{equation}
\label{eq:reflected_noise_derivative}
D_1k_I(z,\cdot)[u]\,m
=
\mathcal R_*
\bigl[-\nabla q(\cdot-z)\cdot u\,dx\bigr].
\end{equation}
\end{enumerate}

In both cases,
\begin{equation}
\label{eq:kernel_derivative_bound}
\|D_1k_X(z,\cdot)[u]\|_{L^1(m)}
\leq
\|\nabla q\|_{L^1}\,|u|,
\end{equation}
where \(k_X=k_{\mathbb T}\) or \(k_X=k_I\), respectively. Moreover,
\begin{equation}
\label{eq:kernel_translation_bound}
\|k_X(z,\cdot)-k_X(z',\cdot)\|_{L^1(m)}
\leq
\|\nabla q\|_{L^1}\,|z-z'|.
\end{equation}
\end{lemma}

\begin{proof}
It is standard that, for \(q\in W^{1,1}\), the translation map
\[
z\longmapsto q(\cdot-z)
\]
is continuously differentiable with values in \(L^1\), with derivative
\[
u\longmapsto-\nabla q(\cdot-z)\cdot u.
\]
The estimate
\eqref{eq:kernel_translation_bound} in the periodic case follows from
the fundamental theorem of calculus along the segment joining \(z\)
and \(z'\).

In the reflected case, use
\eqref{eq:reflected_kernel_pushforward}. Pushforward by a measurable
map is a contraction in total variation:
\[
\|\mathcal R_*\nu\|_{\mathrm{TV}}
\leq
\|\nu\|_{\mathrm{TV}}.
\]
Applying this contraction to the translated densities and their
difference quotients gives,
\eqref{eq:kernel_derivative_bound}, and
\eqref{eq:kernel_translation_bound}. Notice that no differentiability
of the folding map \(\mathcal R\) is used.
\end{proof}

For later use, define
\begin{equation}
\label{eq:gradient_translation_modulus}
\omega_q(\delta)
:=
\sup_{|h|\leq\delta}
\|\nabla q(\cdot-h)-\nabla q\|_{L^1}.
\end{equation}
Since translations are continuous in \(L^1\),
\begin{equation}
\label{eq:gradient_translation_modulus_zero}
\omega_q(\delta)\longrightarrow0
\qquad\text{as }\delta\to0.
\end{equation}
In the reflected case, the corresponding modulus for
\(D_1k_I\) is bounded by \(\omega_q\), again by contraction of
pushforwards in total variation.

\subsection{Sequential random maps and annealed transfer operators}\label{K}

Let \((f_n^\varepsilon)_{n\in\mathbb Z}\) be a family of measurable
maps. In the periodic case,
\[
f_n^\varepsilon:X\longrightarrow\mathbb T^d.
\]
In the reflected case,
\[
f_n^\varepsilon:X\longrightarrow\mathbb R^d,
\]
and we assume that the images of all the maps \(f_n^\varepsilon\), for
all sufficiently small \(\varepsilon\), are contained in a compact
set \(\mathcal K \subseteq \mathbb{R}^d \).

We assume that the perturbation is uniformly differentiable:
there exist measurable maps
\(\dot f_n:X\to\mathbb R^d\) and remainders
\(r_n^\varepsilon:X\to\mathbb R^d\) such that
\begin{equation}
\label{eq:fn_eps}
f_n^\varepsilon
=
f_n^0+\varepsilon\dot f_n+r_n^\varepsilon,
\qquad
\sup_n\|\dot f_n\|_{L^\infty}<\infty,
\qquad
\sup_n
\frac{\|r_n^\varepsilon\|_{L^\infty}}{|\varepsilon|}
\xrightarrow[\varepsilon\to0]{}0.
\end{equation}
In the torus case, this expansion is understood through compatible
periodic lifts to \(\mathbb R^d\).
The corresponding random dynamics is
\[
X_{n+1}
=
f_n^\varepsilon(X_n)+\xi_n
\quad\text{on }\mathbb T^d
\]
in the periodic case, and
\[
X_{n+1}
=
\mathcal R\bigl(f_n^\varepsilon(X_n)+\xi_n\bigr)
\quad\text{on }I^d
\]
in the reflected case.

Writing \(k_X=k_{\mathbb T}\) or \(k_X=k_I\), the transition kernel is
\begin{equation}
\label{eq:general_noise_kernel}
K_n^\varepsilon(x,dy)
=
k_X\bigl(f_n^\varepsilon(x),y\bigr)\,m(dy).
\end{equation}
The associated annealed transfer operator
\(L_n^\varepsilon:L^1(X,m)\to L^1(X,m)\) is
\begin{equation}
\label{eq:annealed_L}
(L_n^\varepsilon\varphi)(y)
:=
\int_X
\varphi(x)\,
k_X\bigl(f_n^\varepsilon(x),y\bigr)\,dm(x).
\end{equation}
Each \(L_n^\varepsilon\) is positive and preserves mass:
\begin{equation}
\label{eq:mass_preserving_noise}
\int_XL_n^\varepsilon\varphi\,dm
=
\int_X\varphi\,dm,
\qquad
\varphi\in L^1(X,m).
\end{equation}
We set
\[
B_s=B_w=L^1(X,m),
\qquad
V_s=
\left\{
\varphi\in L^1(X,m):
\int_X\varphi\,dm=0
\right\}.
\]

Since the noise variables are i.i.d.\ and independent of the past,
the process is a time-inhomogeneous Markov chain. Its time marginals
satisfy
\[
\mu_{n+1}^\varepsilon
=
L_n^\varepsilon\mu_n^\varepsilon.
\]
Accordingly, the natural nonautonomous statistical state is an
equivariant family of probability densities
\(\boldsymbol\mu^\varepsilon=(\mu_n^\varepsilon)_{n\in\mathbb Z}\).

 In the reflected case, the centres of the noise will range in the fixed
compact set \(\mathcal K\subset\mathbb R^d\). 
In this case, we will assume that
\begin{equation}\label{eq:noise_positive_reflected}
\alpha:=
\ell^d\inf_{z\in K}
\operatorname*{ess\,inf}_{y\in I^d}q(y-z)>0.
\end{equation}
This condition is automatic, for example, for every nondegenerate
Gaussian density, since such a density is continuous and strictly
positive and the set
\[
\{y-z:y\in I^d,\ z\in\mathcal K\}
\]
is compact. In this case, one obtains
\begin{equation}
\label{eq:reflected_kernel_lower_bound}
k_I(z,y)\geq\alpha
\qquad
\text{for every \(z\in\mathcal K\) and \(m\)-a.e.\ \(y\in I^d\)}.
\end{equation}

\subsection{Uniform exponential loss of memory}

By \eqref{eq:noise_positive_torus} in the periodic case and by
\eqref{eq:reflected_kernel_lower_bound} in the reflected case, one has
\begin{equation}
\label{eq:noise_minorization}
K_n^\varepsilon(x,A)
\geq
\alpha\,m(A)
\end{equation}
for every \(n\), every sufficiently small \(\varepsilon\), every
\(x\in X\), and every measurable \(A\subset X\). Hence the Doeblin
minorization \eqref{eq:minorization} holds uniformly in \(n\) and
\(\varepsilon\).
By decreasing the minorization constant if necessary,
we may assume $0<\alpha<1$.
By Theorem~\ref{thm:doeblin_ELoM}, for all \(m<n\) and every
\(h\in V_s\),

\begin{equation}
\label{eq:ELoM_noise}
\|L_{n-1}^\varepsilon\cdots L_m^\varepsilon h\|_{L^1}
\leq
(1-\alpha)^{n-m}\|h\|_{L^1}.
\end{equation}
Therefore Assumption~\ref{ass:uniform_ELoM} holds with
\[
C=1,
\qquad
\rho=1-\alpha,
\]
uniformly in \(\varepsilon\).

The same estimate gives existence and uniqueness of an equivariant
family. Fix a probability density \(u\in L^1(X,m)\) and define
\[
\mu_n^{\varepsilon,(j)}
:=
L_{n-1}^\varepsilon\cdots L_{n-j}^\varepsilon u.
\]
If \(j'>j\), then the two pullback iterates can be written as the
images under the same \(j\)-step composition of two probability
densities. Hence
\[
\|\mu_n^{\varepsilon,(j')}
-\mu_n^{\varepsilon,(j)}\|_{L^1}
\leq
2(1-\alpha)^j.
\]
Thus the pullback sequence is Cauchy in \(L^1\), uniformly in \(n\),
and converges to an equivariant family
\(\boldsymbol\mu^\varepsilon\). Uniqueness follows from
\eqref{eq:ELoM_noise}. Moreover,
\[
\|\mu_n^\varepsilon\|_{L^1}=1
\]
for every \(n\) and \(\varepsilon\), so
Assumption~\ref{ass:mu_bound} holds with \(M_s=1\).

\subsection{Differentiability of the transfer operators with respect to the perturbation}

Define
\(\dot L_n:L^1(X,m)\to L^1(X,m)\) by
\begin{equation}
\label{eq:dotL_noise}
(\dot L_n\varphi)(y)
:=
\int_X
\varphi(x)\,
D_1k_X\bigl(f_n^0(x),y\bigr)
[\dot f_n(x)]\,dm(x).
\end{equation}
By Lemma~\ref{lem:noise_kernel_differentiability},
\begin{equation}
\label{eq:dotL_noise_bound}
\|\dot L_n\varphi\|_{L^1}
\leq
\|\nabla q\|_{L^1}
\|\dot f_n\|_{L^\infty}
\|\varphi\|_{L^1}.
\end{equation}
In particular,
\[
\sup_n\|\dot L_n\|_{L^1\to L^1}<\infty.
\]
Since the kernels \(k_X(z,\cdot)\) have total mass one,
differentiation with respect to \(z\) gives
\begin{equation}
\label{eq:dotL_zero_mass}
\int_X\dot L_n\varphi\,dm=0
\qquad
\text{for every }\varphi\in L^1(X,m).
\end{equation}

\begin{lemma}
\label{lem:noise_operator_differentiability}
Under \eqref{eq:fn_eps}, the annealed transfer operators are uniformly
differentiable in operator norm:
\begin{equation}
\label{eq:noise_operator_derivative}
\sup_{n\in\mathbb Z}
\left\|
\frac{L_n^\varepsilon-L_n^0}{\varepsilon}
-\dot L_n
\right\|_{L^1\to L^1}
\xrightarrow[\varepsilon\to0]{}0.
\end{equation}
Moreover,
\begin{equation}
\label{eq:noise_operator_continuity}
\sup_{n\in\mathbb Z}
\|L_n^\varepsilon-L_n^0\|_{L^1\to L^1}
\xrightarrow[\varepsilon\to0]{}0.
\end{equation}
\end{lemma}

\begin{proof}
Set
\[
\Delta_n^\varepsilon
:=
f_n^\varepsilon-f_n^0
=
\varepsilon\dot f_n+r_n^\varepsilon.
\]
Using the fundamental theorem of calculus for the \(L^1\)-valued map
\(z\mapsto k_X(z,\cdot)\), one obtains
\begin{align*}
&
\frac{
k_X(f_n^\varepsilon(x),\cdot)
-
k_X(f_n^0(x),\cdot)
}{\varepsilon}
-
D_1k_X(f_n^0(x),\cdot)[\dot f_n(x)]
\\
&\quad =
D_1k_X(f_n^0(x),\cdot)
\left[
\frac{r_n^\varepsilon(x)}{\varepsilon}
\right]
\\
&\qquad\quad
+
\int_0^1
\Bigl(
D_1k_X(f_n^0(x)+t\Delta_n^\varepsilon(x),\cdot)
-
D_1k_X(f_n^0(x),\cdot)
\Bigr)
\left[
\frac{\Delta_n^\varepsilon(x)}{\varepsilon}
\right]dt.
\end{align*}
By Lemma~\ref{lem:noise_kernel_differentiability} and
\eqref{eq:gradient_translation_modulus},
\begin{align*}
&
\sup_{n,x}
\left\|
\frac{
k_X(f_n^\varepsilon(x),\cdot)
-
k_X(f_n^0(x),\cdot)
}{\varepsilon}
-
D_1k_X(f_n^0(x),\cdot)[\dot f_n(x)]
\right\|_{L^1}
\\
&\quad\leq
\|\nabla q\|_{L^1}
\sup_n
\frac{\|r_n^\varepsilon\|_{L^\infty}}{|\varepsilon|}
\\
&\qquad
+
\omega_q\!\left(
\sup_n\|\Delta_n^\varepsilon\|_{L^\infty}
\right)
\left(
\sup_n\|\dot f_n\|_{L^\infty}
+
\sup_n
\frac{\|r_n^\varepsilon\|_{L^\infty}}{|\varepsilon|}
\right).
\end{align*}
The right-hand side tends to zero by
\eqref{eq:fn_eps} and
\eqref{eq:gradient_translation_modulus_zero}. Integrating against
\(|\varphi(x)|\,dm(x)\) proves
\eqref{eq:noise_operator_derivative}.

Similarly, \eqref{eq:kernel_translation_bound} gives
\[
\|(L_n^\varepsilon-L_n^0)\varphi\|_{L^1}
\leq
\|\nabla q\|_{L^1}
\|f_n^\varepsilon-f_n^0\|_{L^\infty}
\|\varphi\|_{L^1}.
\]
Taking the supremum over \(n\) proves
\eqref{eq:noise_operator_continuity}.
\end{proof}

Let
\[
\boldsymbol\mu=\boldsymbol\mu^0
\]
be the unperturbed equivariant family, and set
\begin{equation}
\label{eq:dotL_along_mu_noise}
g_n:=\dot L_n\mu_n.
\end{equation}
By \eqref{eq:dotL_noise_bound},
\begin{equation}
\label{eq:forcing_noise_bound}
\sup_n\|g_n\|_{L^1}
\leq
\|\nabla q\|_{L^1}
\sup_n\|\dot f_n\|_{L^\infty}
<\infty.
\end{equation}
Moreover, Lemma~\ref{lem:noise_operator_differentiability} gives
\begin{equation}
\label{eq:strong_derivative_noise}
\lim_{\varepsilon\to0}
\sup_{n\in\mathbb Z}
\left\|
\frac{
L_n^\varepsilon\mu_n-L_n^0\mu_n
}{\varepsilon}
-g_n
\right\|_{L^1}
=0.
\end{equation}
Thus Assumption~\ref{ass:strong_derivative} holds with
\(B_s=B_w=L^1(X,m)\).

In the periodic case, the derivative operator has the explicit form
\begin{equation}
\label{eq:dotL_torus_explicit}
(\dot L_n\varphi)(y)
=
-
\int_{\mathbb T^d}
\varphi(x)\,
\nabla q\bigl(y-f_n^0(x)\bigr)
\cdot\dot f_n(x)\,dm(x).
\end{equation}
In the reflected case, the derivative is given by
\eqref{eq:dotL_noise} together with
\eqref{eq:reflected_noise_derivative}.

\subsection{Linear response}
We are finally ready to state a first general result on linear response for sequential random systems with additive noise.
\begin{theorem}[Linear response for compact random maps with additive noise]
\label{thm:LR_noise}
Assume the noise hypotheses stated above and the uniform
differentiability condition \eqref{eq:fn_eps}. Let
\(\boldsymbol\mu^\varepsilon\) be the equivariant family associated
with \((L_n^\varepsilon)_{n\in\mathbb Z}\). Then the hypotheses of
Theorem~\ref{thm:LR_strong} are satisfied with
\[
B_s=B_w=L^1(X,m).
\]
In particular,
\[
\boldsymbol\eta
=
\lim_{\varepsilon\to0}
\frac{
\boldsymbol\mu^\varepsilon-\boldsymbol\mu^0
}{\varepsilon}
\]
exists in
\(
\mathcal B_w
=
\ell^\infty(\mathbb Z;L^1(X,m))
\) and it satisfies \eqref{eq:LR_neumann} and \eqref{eq:LR_series}. 
\end{theorem}

\begin{proof}
Mass preservation follows from
\eqref{eq:mass_preserving_noise}. Since \(B_w=L^1(X,m)\) and every
\(L_n^\varepsilon\) is Markov, the weak power-boundedness assumption
\emph{(A2)} holds with \(M_w=1\). Uniform exponential loss of memory
on the zero-mass subspace follows from
\eqref{eq:ELoM_noise}, with \(C=1\) and
\(\rho=1-\alpha\). The equivariant family satisfies
\(
\sup_{\varepsilon,n}
\|\mu_n^\varepsilon\|_{L^1}=1.
\)
Finally, Assumption~\ref{ass:strong_derivative} follows from
\eqref{eq:noise_operator_derivative},
\eqref{eq:noise_operator_continuity}, and
\eqref{eq:strong_derivative_noise}. Therefore
Theorem~\ref{thm:LR_strong} applies, giving the statement.
\end{proof}


\subsection{Reflected Euler--Maruyama schemes for dissipative nonautonomous SDEs}
\label{subsec:LR_reflected_EM}

In this section we verify the hypotheses of Theorem ~\ref{thm:LR_noise} 
for reflected Euler-Maruyama discretizations of
dissipative nonautonomous stochastic differential
equations with nondegenerate additive Gaussian noise.
For a fixed time step and a fixed reflecting cube,
we establish uniform Doeblin minorization and
differentiability with respect to the perturbation
parameter. This construction also underlies the
discretized Ghil-Sellers model considered in Section~ \ref{subsec:LR_Ghil_Sellers}.

Let $\varepsilon\in[0,\varepsilon_0)$ and consider on
$\mathbb R^d$ the nonautonomous SDE
\begin{equation}
\label{eq:time_dependent_SDE}
dY_t^\varepsilon
=
b^\varepsilon(t,Y_t^\varepsilon)\,dt
+
\Sigma\,dW_t,
\end{equation}
where $W_t$ is a standard $d$-dimensional Brownian motion and
$\Sigma\in\mathbb R^{d\times d}$ is constant. We write
\[
A:=\Sigma\Sigma^{\mathsf T}
\]
and assume uniform ellipticity:
\begin{equation}
\label{eq:uniform_ellipticity_EM}
\lambda_-|\xi|^2
\leq
\langle A\xi,\xi\rangle
\leq
\lambda_+|\xi|^2
\qquad
\text{for every }\xi\in\mathbb R^d,
\end{equation}
for some constants $0<\lambda_-\leq\lambda_+<\infty$.

We assume that the drift is measurable in time and
locally Lipschitz in the state variable, uniformly
with respect to time and the perturbation parameter.
More precisely, for every $R>0$ there exists
$L_R<\infty$ such that
\begin{equation}\label{eq:drift_global_Lipschitz}
|b^\varepsilon(t,x)-b^\varepsilon(t,y)|
\leq L_R|x-y|
\qquad
\text{for }|x|,|y|\leq R,
\end{equation}
for every $t\in\mathbb R$ and
$\varepsilon\in[0,\varepsilon_0)$.
We also assume that there exist constants $c_0,c_1>0$,
independent of $t$ and $\varepsilon$, such that
\begin{equation}\label{eq:drift_dissipativity}
\langle b^\varepsilon(t,x),x\rangle
\leq c_0-c_1|x|^2
\end{equation}
for every $t\in\mathbb R$, $x\in\mathbb R^d$,
and $\varepsilon\in[0,\varepsilon_0)$.

These assumptions also imply a uniform bound at the origin.
Indeed, if $b^\varepsilon(t,0)\neq0$, set
$e=b^\varepsilon(t,0)/|b^\varepsilon(t,0)|$.
The local Lipschitz estimate on the unit ball and
dissipativity give
\[
\begin{aligned}
|b^\varepsilon(t,0)|
&=
\langle b^\varepsilon(t,0)-b^\varepsilon(t,e),e\rangle
+\langle b^\varepsilon(t,e),e\rangle\\
&\leq L_1+c_0-c_1
\leq B,
\qquad B:=L_1+c_0.
\end{aligned}
\]
The same bound holds trivially when $b^\varepsilon(t,0)=0$.

These assumptions guarantee global well-posedness
of the SDE and uniform in time second-moment bounds
for solutions with square integrable initial data.
For the linear response analysis of the reflected
scheme below, we only need uniform boundedness of
the drift on the prescribed compact domain and
the uniform parameter expansion stated next.


We assume that the parameter dependence of the drift is
differentiable locally uniformly in the state variable. More
precisely, there exist a measurable vector field
$\dot b:\mathbb R\times\mathbb R^d\to\mathbb R^d$ and remainders
$r^\varepsilon$ such that
\begin{equation}
\label{eq:drift_parameter_expansion}
b^\varepsilon(t,x)
=
b^0(t,x)
+
\varepsilon\dot b(t,x)
+
r^\varepsilon(t,x),
\end{equation}
and, for every $R>0$,
\begin{equation}
\label{eq:drift_parameter_uniformity}
\sup_{\substack{t\in\mathbb R\\ |x|\leq R}}
|\dot b(t,x)|
<\infty,
\qquad
\sup_{\substack{t\in\mathbb R\\ |x|\leq R}}
\frac{|r^\varepsilon(t,x)|}{|\varepsilon|}
\xrightarrow[\varepsilon\to0]{}0.
\end{equation}
Only these local uniform estimates will be needed after restricting
the numerical dynamics to a compact domain.

\paragraph{Euler--Maruyama dynamics on $\mathbb R^d$.}
Fix a time step $h>0$ and let
\[
t_k=t_0+kh,
\qquad
k\in\mathbb Z.
\]
The Euler--Maruyama discretization of
\eqref{eq:time_dependent_SDE} is
\begin{equation}
\label{eq:EM_unbounded}
Y_{k+1}^\varepsilon
=
Y_k^\varepsilon
+
h\,b^\varepsilon(t_k,Y_k^\varepsilon)
+
\sqrt h\,\Sigma\xi_k,
\end{equation}
where $(\xi_k)_{k\in\mathbb Z}$ are independent standard Gaussian
vectors in $\mathbb R^d$. Introducing the deterministic Euler maps
\begin{equation}
\label{eq:Euler_maps}
F_k^\varepsilon(x)
:=
x+h\,b^\varepsilon(t_k,x),
\end{equation}
we may write
\[
Y_{k+1}^\varepsilon
=
F_k^\varepsilon(Y_k^\varepsilon)
+
\sqrt h\,\Sigma\xi_k.
\]


\paragraph{Reflected Euler--Maruyama dynamics on a cube.}
Let $I=[a,b]$ be a large closed interval and set
\[
X=I^d.
\]
Let $\mathcal R:\mathbb R^d\to I^d$ be the componentwise folding map
defined in \eqref{eq:folding_map_1d}. We consider the reflected
Euler--Maruyama chain
\begin{equation}
\label{eq:EM_reflected}
X_{k+1}^\varepsilon
=
\mathcal R\left(
F_k^\varepsilon(X_k^\varepsilon)
+
\sqrt h\,\Sigma\xi_k
\right).
\end{equation}
Thus the deterministic centre of the noise at $x\in I^d$ is
$F_k^\varepsilon(x)$, while the complete random step is returned to
$I^d$ by reflection.

Choose $R>0$ such that $I^d\subseteq\overline B(0,R)$.

The uniform local Lipschitz condition and the bound
$|b^\varepsilon(t,0)|\leq B$ established above imply
\[
M_I:=
\sup_{\substack{k\in\mathbb Z,\;
                 \varepsilon\in[0,\varepsilon_0)\\
                 x\in I^d}}
|b^\varepsilon(t_k,x)|
\leq B+L_RR<\infty.
\]
Only a bound on the fixed reflecting cube, uniform in time
and in the perturbation parameter, is needed here.
No global boundedness of the drift on $\mathbb{R}^d$
is assumed.
Hence all deterministic centres belong to the compact set
\begin{equation}
\label{eq:EM_centres_compact_set}
\mathcal K_h
:=
I^d+\overline{B(0,hM_I)},
\qquad
F_k^\varepsilon(I^d)\subset\mathcal K_h.
\end{equation}

The random vector $\sqrt h\,\Sigma\xi_k$ has Gaussian density
\(
q_h\in W^{1,1}
\)
and $q_h(z)>0$ for every $z\in\mathbb R^d$. Let
$k_{I,h}(z,\cdot)$ be the reflected Gaussian kernel obtained from
$q_h$ as in \eqref{eq:reflected_kernel_pushforward}

\begin{equation}
\label{eq:EM_kernel_lower_bound}
k_{I,h}\bigl(F_k^\varepsilon(x),y\bigr)
\geq
\alpha_h>0
\end{equation}
uniformly in $k\in\mathbb Z$, 
$\varepsilon \in [0,\varepsilon_0)$, $x\in I^d$, and for
$m$-a.e.\ $y\in I^d$.

The annealed transfer operator associated with
\eqref{eq:EM_reflected} is
\begin{equation}
\label{eq:EM_transfer_operator}
(L_k^\varepsilon\varphi)(y)
=
\int_{I^d}
\varphi(x)\,
k_{I,h}\bigl(F_k^\varepsilon(x),y\bigr)\,dm(x).
\end{equation}

The assumptions above ensure well-posedness and moment bounds
for the underlying SDE. The linear response result for the
reflected numerical scheme requires only the following
conditions on the fixed reflecting cube.

\begin{theorem}[Linear response for reflected Euler--Maruyama schemes]
\label{thm:LR_reflected_EM}
Assume \eqref{eq:uniform_ellipticity_EM}, and fix a time
step $h>0$ and a reflecting cube $I^d$.
Suppose that the measurable drifts satisfy
\[
b^\varepsilon(t,x)
=
b^0(t,x)+\varepsilon\dot b(t,x)+r_\varepsilon(t,x),
\qquad (t,x)\in\mathbb R\times I^d,
\]
where
\[
\sup_{\substack{t\in\mathbb R\\x\in I^d}}
\bigl(|b^0(t,x)|+|\dot b(t,x)|\bigr)<\infty,
\qquad
\sup_{\substack{t\in\mathbb R\\x\in I^d}}
\frac{|r_\varepsilon(t,x)|}{\varepsilon}
\xrightarrow[\varepsilon\downarrow0]{}0.
\]
For every sufficiently small $\varepsilon\geq0$, the
reflected Euler--Maruyama chain \eqref{eq:EM_reflected}
has a unique equivariant family of probability densities
$\boldsymbol\mu^\varepsilon$, and
\[
\boldsymbol\eta
=
\lim_{\varepsilon\downarrow0}
\frac{\boldsymbol\mu^\varepsilon-\boldsymbol\mu^0}
{\varepsilon}
\quad\text{exists in }
\ell^\infty\bigl(\mathbb Z;L^1(I^d,m)\bigr).
\]
The derivative operator is
\begin{equation}
\label{eq:EM_derivative_operator}
(\dot L_k\varphi)(y)
=
\int_{I^d}
\varphi(x)\,
D_1k_{I,h}\bigl(F_k^0(x),y\bigr)
\bigl[h\dot b(t_k,x)\bigr]\,dm(x),
\end{equation}
where $D_1$ is understood in the $L^1$ sense specified
in Lemma~\ref{lem:noise_kernel_differentiability}.
With $L_k=L_k^0$, $\mu_k=\mu_k^0$, and
$g_k=\dot L_k\mu_k$, the response is given by
\eqref{eq:LR_series}, with absolute convergence in $L^1$,
uniformly in time.
\end{theorem}

\begin{proof}
The assumptions imply that $b^\varepsilon$ is uniformly
bounded on $\mathbb R\times I^d$ for sufficiently small
$\varepsilon$. Hence the Euler maps
$F_k^\varepsilon(x)=x+h b^\varepsilon(t_k,x)$
take values in a common compact set.
The nondegenerate Gaussian density $q_h$ belongs to
$W^{1,1}(\mathbb R^d)$ and is strictly positive, so the
associated reflected kernels satisfy a uniform Doeblin
minorization. Moreover,
\[
F_k^\varepsilon
=
F_k^0+\varepsilon h\dot b(t_k,\cdot)
+h r_\varepsilon(t_k,\cdot)
\]
gives the required uniform parameter expansion on $I^d$.
The conclusion follows from Theorem~\ref{thm:LR_noise}
with $X=I^d$, $q=q_h$, and $f_k^\varepsilon=F_k^\varepsilon$.
\end{proof}

\begin{remark}
\label{rem:EM_reflection_interpretation}
The reflection in \eqref{eq:EM_reflected} is imposed at the level of
the discrete numerical scheme. Thus the system should be
viewed as a reflected Euler--Maruyama approximation of
\eqref{eq:time_dependent_SDE}, rather than as the Euler discretization
of a continuous reflected SDE. This is also a natural implementation
in numerical simulations when the state variables are required to
remain in a prescribed admissible interval.
\end{remark}

\section{Application to  the Ghil--Sellers model}
\label{subsec:LR_Ghil_Sellers}
{We now provide a physically-relevant example to illustrate some of the key aspects of our findings.}
We apply Theorem~\ref{thm:LR_reflected_EM}
to a reflected finite-dimensional stochastic
discretization of the nonautonomous Ghil--Sellers
energy balance model.
The model itself leads to a stochastic partial differential equation representing the evolution of the longitudinally-averaged surface temperature of the climate systems at different latitudes. 
The Ghil-Sellers energy balance model  (GSEBM) \cite{Sellers,Ghil1976} is one of the foundational models in climate science and has played a major role in characterising the multistability of the Earth system \cite{Bodai2015,Ghil2020}. {We consider here the stochastic version used in \cite{LucariniChekroun2024,Lucarini2026}.} The model describes the processes of energy absorption, emission, and energy redistribution across the latitudes of the climate system.
The  model
describes the evolution of the zonally averaged surface temperature
\[
T=T(t,x),
\]
where \(t\geq0\) denotes time and
\[
x=\frac{2\phi}{\pi}\in[-1,1]
\]
is the normalized latitude corresponding to the geographical latitude
\(\phi\in[-\pi/2,\pi/2]\). Thus \(x=0\) represents the equator, while
\(x=\pm1\) correspond to the two poles.

Introducing the geometric weight
\[
\omega(x):=\cos\left(\frac{\pi x}{2}\right),
\]
which accounts for the variation of the length of latitude circles, the
model can be written as
\begin{align}
\partial_tT(t,x)
&=
\frac{1}{c(x)}
\left(\frac{2}{\pi}\right)^2
\frac{1}{\omega(x)}
\partial_x\left(
\omega(x)k(x,T)\partial_xT
\right)
\nonumber\\
&\quad
+
\frac{1}{c(x)}
Q(t,x)\bigl(1-\alpha(x,T)\bigr)
\nonumber\\
&\quad
-
\frac{\sigma}{c(x)}
T^4
\left(
1-m(t)\tanh(c_3T^6)
\right)
+
\eta_0\eta(t,x).
\label{eq:GSEBM_PDE_application}
\end{align}

Here $c(x)>0$ is the effective heat capacity of the local
atmosphere--land--ocean column. The first term describes
meridional heat transport, with effective diffusivity
$k(x,T)>0$ accounting for sensible and latent heat transport;
its temperature dependence incorporates the
Clausius--Clapeyron relation. The boundary conditions
$\partial_xT(t,\pm1)=0$ impose no heat flux through the poles.

The absorbed solar radiation is determined by the incoming
irradiance $Q(t,x)\geq0$ and the planetary albedo
$\alpha(x,T)\in[0,1]$. The decrease of albedo with temperature
produces the ice--albedo feedback underlying the coexistence
of warm and snowball states. Outgoing radiation follows the
Stefan--Boltzmann law, with constant $\sigma>0$ and emissivity
modified by $1-m(t)\tanh(c_3T^6)$, where $c_3>0$.
Increasing $m(t)$ reduces outgoing radiation and represents
stronger greenhouse forcing.

The stochastic term $\eta_0\eta(t,x)$ represents unresolved
fast atmospheric and oceanic variability in the spirit of
Hasselmann's stochastic climate modelling. Below, it is
represented by nondegenerate additive Gaussian noise.
The explicit time dependence of $Q$ and $m$ makes the model
nonautonomous, motivating the study of equivariant families
rather than a stationary measure.

Following standard numerical approaches to climate modelling
\cite{Neelin_2010,Drake2014,Ghil2020}, we discretize latitude
and then apply the Euler--Maruyama scheme \cite{Kloeden2011},
with reflection on a compact temperature domain.
We study the response of this numerical model to
time-dependent greenhouse perturbations representing
changes in CO$_2$ forcing, as in \cite{Lucarini2026}.

\subsection{Finite-dimensional stochastic model}
Following what is usually done to numerically study the evolution of the Ghil-Sellers model,   we discretize the space, obtaining a system of stochastic differential equations.
We use a conservative finite-volume discretization of the latitude
variable. Let
\[
-1=x_{1/2}<x_{3/2}<\cdots<x_{N+1/2}=1
\]
be a partition of $[-1,1]$, let
\[
C_i=[x_{i-1/2},x_{i+1/2}],
\qquad
x_i\in C_i,
\]
and define the positive cell weights
\[
w_i:=\int_{C_i}\omega(x)\,dx,
\qquad
c_i:=c(x_i).
\]
The vector
\[
u=(u_1,\ldots,u_N)\in\mathbb R^N
\]
represents the discrete temperature profile, with
$u_i\approx T(t,x_i)$.

For $1\leq i\leq N-1$, define the discrete meridional heat flux
\begin{equation}
\label{eq:GS_discrete_flux}
J_{i+1/2}(u)
:=
\omega(x_{i+1/2})
k_{i+1/2}(u)
\frac{u_{i+1}-u_i}{x_{i+1}-x_i},
\end{equation}
where $k_{i+1/2}(u)>0$ is a consistent discretization of the
temperature-dependent diffusivity $k(x,T)$. The no-flux boundary
conditions are represented by
\[
J_{1/2}(u)=J_{N+1/2}(u)=0.
\]
The discrete diffusion term is
\begin{equation}
\label{eq:GS_discrete_diffusion}
(\mathcal D_N(u))_i
:=
\left(\frac{2}{\pi}\right)^2
\frac{
J_{i+1/2}(u)-J_{i-1/2}(u)
}{
c_iw_i
}.
\end{equation}

For the weighted scalar product
\[
\langle u,v\rangle_{c,w}
:=\sum_{i=1}^{N}c_iw_i u_iv_i,
\]
discrete summation by parts, using the vanishing boundary
fluxes, gives
\begin{equation}\label{eq:GS_diffusion_dissipation}
\begin{aligned}
\langle u,D_N(u)\rangle_{c,w}
&=
-\left(\frac{2}{\pi}\right)^2
\sum_{i=1}^{N-1}
\omega(x_{i+1/2})k_{i+1/2}(u)
\frac{(u_{i+1}-u_i)^2}{x_{i+1}-x_i}\leq0.
\end{aligned}
\end{equation}
Thus the discrete meridional transport is dissipative
in the natural heat-capacity-weighted scalar product.

%
%


We set
\[
Q_i(t):=Q(t,x_i),
\qquad
\alpha_i(s):=\alpha(x_i,s).
\]

Since physical temperatures are expressed in Kelvin, the relevant
temperature interval is contained in $(0,\infty)$. In order to define
an auxiliary dissipative equation on the whole space $\mathbb R^N$,
we extend the Stefan--Boltzmann law by means of the signed power
\begin{equation}
\label{eq:GS_radiation_extension}
\mathcal E(s):=s|s|^3.
\end{equation}
Then
\[
\mathcal E(s)=s^4
\qquad\text{for }s\geq0,
\]
so this extension does not modify the physical outgoing-radiation law,
while
\[
s\mathcal E(s)=|s|^5
\qquad\text{for every }s\in\mathbb R.
\]
Consequently, the extended outgoing-radiation term remains coercive
also for the nonphysical negative-temperature states of the auxiliary
equation on $\mathbb R^N$.

We consider a family of greenhouse parameters
\begin{equation}
\label{eq:GS_CO2_perturbation}
m^\varepsilon(t)
=
m(t)
+
\varepsilon\dot m(t)
+
r_m^\varepsilon(t),
\end{equation}
where
\begin{equation}
\label{eq:GS_CO2_uniform_differentiability}
\sup_{t\in\mathbb R}|\dot m(t)|<\infty,
\qquad
\sup_{t\in\mathbb R}
\frac{|r_m^\varepsilon(t)|}{|\varepsilon|}
\xrightarrow[\varepsilon\to0]{}0.
\end{equation}
We assume that there exists $m_*<1$ such that
\begin{equation}
\label{eq:GS_greenhouse_bound}
0\leq m^\varepsilon(t)\leq m_*
\end{equation}
for every $t\in\mathbb R$ and every sufficiently small $\varepsilon$.

The finite-dimensional drift
$F_N^\varepsilon:\mathbb R\times\mathbb R^N\to\mathbb R^N$ is
defined componentwise by
\begin{align}
(F_N^\varepsilon(t,u))_i
&=
(\mathcal D_N(u))_i
+
\frac{Q_i(t)}{c_i}
\bigl(1-\alpha_i(u_i)\bigr)
\nonumber\\
&\quad
-
\frac{\sigma}{c_i}
\mathcal E(u_i)
\left(
1-m^\varepsilon(t)\tanh(c_3u_i^6)
\right).
\label{eq:GS_finite_dimensional_drift}
\end{align}
Let $\Sigma_N\in\mathbb R^{N\times N}$ be a constant noise matrix such
that
\begin{equation}
\label{eq:GS_uniform_ellipticity}
A_N:=\Sigma_N\Sigma_N^{\mathsf T}
\geq
\lambda_N I_N
\end{equation}
for some $\lambda_N>0$. The spatially discretized stochastic
Ghil--Sellers model is the nonautonomous SDE
\begin{equation}
\label{eq:GS_finite_dimensional_SDE}
dU_t^\varepsilon
=
F_N^\varepsilon(t,U_t^\varepsilon)\,dt
+
\Sigma_N\,dW_t.
\end{equation}

We assume that $Q_i$ is uniformly bounded in time, that
$0\leq\alpha_i\leq1$, and that the functions entering
$\mathcal D_N$ and $\alpha_i$ are locally Lipschitz. Then
$F_N^\varepsilon$ is locally Lipschitz in $u$, uniformly on bounded
sets with respect to $t$ and small $\varepsilon$.

By \eqref{eq:GS_diffusion_dissipation}, the diffusion term
is nonpositive in the weighted scalar product.
The absorbed-radiation term contributes at most linear
growth in $u$, whereas
$u_iE(u_i)=|u_i|^5$ and
$1-m^\varepsilon(t)\tanh(c_3u_i^6)\geq1-m_*>0$.
Consequently, there exist $C_N,\gamma_N>0$, independent
of $t$ and small $\varepsilon$, such that
\begin{equation}\label{eq:GS_dissipativity}
\langle F_N^\varepsilon(t,u),u\rangle_{c,w}
\leq C_N-\gamma_N\|u\|_{c,w}^2,
\qquad
\|u\|_{c,w}^2:=\langle u,u\rangle_{c,w}.
\end{equation}
Together with local Lipschitz continuity, this estimate
makes $V_N(u)=1+\|u\|_{c,w}^2$ a Lyapunov function,
yielding global well-posedness of
\eqref{eq:GS_finite_dimensional_SDE}
and uniform-in-time second-moment bounds for initial
data with finite second moment.

\subsection{Euler--Maruyama discretization and compact reflection}

Now, to restrict our random dynamical system to a compact domain of physically meaningful temperatures, we fix a bounded domain and consider reflecting boundary conditions.
Fix a time step $h>0$, set $t_k=t_0+kh$, and define the deterministic
Euler maps
\begin{equation}
\label{eq:GS_Euler_map}
\Phi_{k,N}^\varepsilon(u)
:=
u+hF_N^\varepsilon(t_k,u).
\end{equation}

Let $(\xi_k)_{k\in\mathbb Z}$ be independent standard
Gaussian vectors in $\mathbb R^N$.

Let
\[
I_T=[T_-,T_+]
\]
be a large closed interval containing the physically relevant
temperature range, and set
\[
X_N:=I_T^N.
\]
Let
\[
\mathcal R_N:\mathbb R^N\to I_T^N
\]
be the componentwise folding map introduced in
\eqref{eq:folding_map_1d}. We consider the compact reflected dynamics
\begin{equation}
\label{eq:GS_EM_reflected}
X_{k+1}^{\varepsilon,N}
=
\mathcal R_N\left(
\Phi_{k,N}^\varepsilon(X_k^{\varepsilon,N})
+
\sqrt h\,\Sigma_N\xi_k
\right).
\end{equation}
This is a time-dependent random dynamical system on $I_T^N$. It is
also the natural compact numerical scheme obtained by applying an
Euler--Maruyama step and reflecting any temperature component that
leaves the prescribed admissible interval.


Let $q_{N,h}$ denote the density of the Gaussian increment
$\sqrt{h}\,\Sigma_N\xi_k$, whose covariance is $hA_N$.

Since $q_{N,h}\in W^{1,1}(\mathbb R^N)$ and is strictly positive, the
reflected transition kernel satisfies a uniform Doeblin minorization
on $I_T^N$. More precisely,
\begin{equation}
\label{eq:GS_Doeblin_constant}
\alpha_{N,h}
:=
|I_T|^N
\min_{\substack{z\in\mathcal K_{N,h}\\y\in I_T^N}}
q_{N,h}(y-z)
>0.
\end{equation}

\subsection{Climate Change in the Compact Model}
Now we  check that the perturbations considered satisfy the assumptions needed to apply our general results.
The perturbation \eqref{eq:GS_CO2_perturbation} affects only the
outgoing-radiation term. Differentiating
\eqref{eq:GS_finite_dimensional_drift} at $\varepsilon=0$ gives
\begin{equation}
\label{eq:GS_drift_derivative}
(G_N(t,u))_i
:=
\left.
\partial_\varepsilon
(F_N^\varepsilon(t,u))_i
\right|_{\varepsilon=0}
=
\frac{\sigma}{c_i}
\dot m(t)\,
\mathcal E(u_i)
\tanh(c_3u_i^6).
\end{equation}

$u\in I_T^N$ and $k\in\mathbb Z$,
%

By \eqref{eq:GS_CO2_uniform_differentiability}, the drift
$F_N^\varepsilon$ admits a parameter expansion with
derivative $G_N$, uniformly on $\mathbb R\times I_T^N$.
This verifies the differentiability assumption of
Theorem~\ref{thm:LR_reflected_EM}.


Let $m_N$ denote normalized Lebesgue measure on $I_T^N$, and let
\[
L_{k,N}^\varepsilon:
L^1(I_T^N,m_N)\longrightarrow L^1(I_T^N,m_N)
\]
be the annealed transfer operator associated with
\eqref{eq:GS_EM_reflected}. Let
\[
\boldsymbol\mu^{\varepsilon,N}
=
(\mu_k^{\varepsilon,N})_{k\in\mathbb Z}
\]
be its unique equivariant family.

\begin{theorem}[Linear response of the discretized Ghil--Sellers model]
\label{thm:LR_compact_Ghil_Sellers}
Assume the regularity and boundedness conditions above,
together with \eqref{eq:GS_uniform_ellipticity},
\eqref{eq:GS_CO2_uniform_differentiability}, and
\eqref{eq:GS_greenhouse_bound}.
For fixed spatial discretization, time step $h>0$, and
reflecting cube $I_T^N$, the equivariant family of
\eqref{eq:GS_EM_reflected} is differentiable with respect
to the CO$_2$-type perturbation \eqref{eq:GS_CO2_perturbation}:
\[
\boldsymbol\eta^N
=
\lim_{\varepsilon\to0}
\frac{\boldsymbol\mu^{\varepsilon,N}
      -\boldsymbol\mu^{0,N}}{\varepsilon}
\quad\text{in }\ell^\infty\bigl(\mathbb Z;L^1(I_T^N,m_N)\bigr).
\]

Write $L_{k,N}=L_{k,N}^0$ and $\mu_k^N=\mu_k^{0,N}$.
Let $k_{I_T,N,h}$ be the reflected kernel associated
with the Gaussian density $q_{N,h}$, and define
\begin{equation}
\label{eq:GS_derivative_transfer_operator}
(\dot L_{k,N}\varphi)(y)
=
\int_{I_T^N}
\varphi(u)\,
D_1k_{I_T,N,h}\bigl(\Phi_{k,N}^0(u),y\bigr)
\bigl[hG_N(t_k,u)\bigr]\,dm_N(u),
\end{equation}
where $D_1$ is the Fr\'echet derivative of
$z\mapsto k_{I_T,N,h}(z,\cdot)$ as an
$L^1(I_T^N,m_N)$-valued map.
With $g_{k,N}:=\dot L_{k,N}\mu_k^N$, the response is
\begin{equation}
\label{eq:GS_response_series}
\eta_{k,N}
=
g_{k-1,N}
+
\sum_{j=1}^{\infty}
L_{k-1,N}L_{k-2,N}\cdots L_{k-j,N}
g_{k-j-1,N}.
\end{equation}
The series converges absolutely in $L^1(I_T^N,m_N)$,
uniformly in $k$.
\end{theorem}

{
\begin{proof}
We apply Theorem~\ref{thm:LR_reflected_EM} with
$d=N$, $I=I_T$, $b^\varepsilon=F_N^\varepsilon$,
and $\Sigma=\Sigma_N$.
Uniform ellipticity follows from
\eqref{eq:GS_uniform_ellipticity}.
The coefficient assumptions imply that $F_N^0$ and $G_N$
are uniformly bounded on $\mathbb R\times I_T^N$, while
\eqref{eq:GS_CO2_perturbation}--%
\eqref{eq:GS_CO2_uniform_differentiability}
give the uniform drift expansion with $\dot b=G_N$.
The conclusion follows, with
\eqref{eq:GS_derivative_transfer_operator} obtained from
\eqref{eq:EM_derivative_operator}.
\end{proof}
}

%
%
%

As a consequence, every bounded observable
$\Psi:I_T^N\to\mathbb R$ has a linear response:
\begin{equation}
\label{eq:GS_observable_response}
\left.
\frac{d}{d\varepsilon}
\int_{I_T^N}
\Psi(u)\,\mu_k^{\varepsilon,N}(u)\,dm_N(u)
\right|_{\varepsilon=0}
=
\int_{I_T^N}
\Psi(u)\,\eta_{k,N}(u)\,dm_N(u).
\end{equation}
In particular, this applies to the discretized global mean
temperature
\[
\overline T_N(u)
:=
\frac{
\sum_{i=1}^Nw_i u_i
}{
\sum_{i=1}^Nw_i
},
\]
which is bounded on $I_T^N$. Thus
\eqref{eq:GS_observable_response} gives the first-order response of the
time-dependent expected global mean temperature to the prescribed
CO$_2$-modulation $\dot m(t)$.

\begin{remark}
\label{rem:GS_compact_model_interpretation}
The compactification is introduced at the level of the numerical
scheme. It does not identify high and low temperatures, as a torus
compactification would do, but reflects exceptional numerical
excursions at the boundary of a large physically admissible
temperature interval. This is a natural procedure in simulations and
allows the linear-response problem to be treated in the unweighted
space $L^1(I_T^N,m_N)$. No claim is made here concerning the limit as
$T_-\to-\infty$, $T_+\to+\infty$, $h\to0$, or $N\to\infty$.
\end{remark}

\section{Discussion and Conclusions}\label{conclusions}

We have developed an abstract linear response framework
for discrete-time deterministic and random sequential systems
with periodic or aperiodic reference dynamics.
Under uniform regularity, differentiability, and exponential
memory-loss assumptions, a global transfer operator on
sequences of measures yields a resolvent formulation and
causal response formulas that are uniform in time.
This provides rigorous foundations for the response formulas
proposed in \cite{Lucarini2026} and complements results
for time-periodic stochastic dynamics \cite{Branicki2021}.

We apply the framework to sequential expanding maps and
random maps with additive noise satisfying a uniform
Doeblin condition. The latter include reflected
Euler--Maruyama discretizations of dissipative nonautonomous
SDEs with nondegenerate additive Gaussian noise, for fixed
time steps and compact reflecting domains.
As a climate application, we establish linear response
to time-dependent greenhouse perturbations for the
Ghil--Sellers model discretized in latitude and time,
with reflection on a prescribed temperature cube.

A subsequent work \cite{GalatoloLucariniNisoli2026} further connects this framework to finite-state Markov modeling.
Under suitable assumptions on memory loss, regularization,
and the approximation scheme, Ulam-type reductions
approximate equivariant families, and their linear response
converges uniformly in time to the response of the original system.
This robustness of linear response under approximation
strengthens the motivation for its study in complex
real-world systems, whose statistical properties are
typically investigated through reduced models.

A further direction is to apply the framework to sequential
uniformly hyperbolic systems using anisotropic Banach spaces
\cite{liverani2006}. This requires establishing the appropriate
uniform memory-loss, regularity, and differentiability
estimates for their transfer operators.

The relation between memory loss, statistical sensitivity,
and tipping also deserves investigation.
Connections between spectral gap closure, response theory,
and tipping in autonomous systems
\cite{LucariniChekroun2023,LucariniChekroun2024,Lenton.tip.08,Lenton2012}
motivate studying how the nonautonomous framework can inform
the analysis of rate-induced \cite{Ashwin2012,Panahi2023}
and phase-induced tipping \cite{Alkhayuon2021},
with potential applications to climate and ecosystem
resilience \cite{Abrams2025,Hastings2026}.

Finally, the response operators provide a starting point
for optimal response, or linear request: finding
perturbations of minimal cost that produce a prescribed
statistical change.
Extending the existing autonomous theory
\cite{Antown2018,Antown2022,Gutierrez2025,FG2025,
FroylandPhalempin2025OptimalLinearResponse,Galatolo_2025,
DGJ,Galatolo_2017,Kloeckner2018LinearRequestProblem}
to nonautonomous reference dynamics would turn these
response formulas into tools for statistical control
of time-dependent systems.

\appendix 

\section{Uniform exponential loss of memory for composition of operators satisfying a common Lasota Yorke inequality. }\label{appendixa}

In this section, we show a relatively simple and general argument (see \cite{GF} or \cite{Galatolo2015} for similar ones) that
establishes  exponential  loss of memory for a sequential composition of
operators which are nearby { or slowly varying} in a mixed topology. 
{We first consider sequential compositions of operators which are all close, in
the mixed strong--weak topology, to a fixed reference operator \(L_0\) with
convergence to equilibrium. We then show that the same argument applies to
slowly varying families, provided the convergence-to-equilibrium assumption is
uniformly satisfied along the family.}

Let $B_{w}$ and $B_{s}$ be normed vector spaces of { signed
measures on $X$.} Suppose $(B_{s},||~||_{s})\subseteq $ $(B_{w},||~||_{w})$
and $||~||_{s}\geq ||~||_{w}$. Let us consider a sequence of bounded linear Markov
operators $\{L_{i}\}_{i\in \mathbb{N}}:B_{s}\rightarrow B_{s}.$ \ We will
suppose furthermore that the following assumptions are satisfied by the $%
L_{i}$:

\begin{itemize}
\item[$ML1$] The operators $L_{i }$ satisfy a common "one step" Lasota-Yorke
inequality. There are constants $B,\lambda _{1}\geq 0$ with $\lambda _{1}<1$, $B\geq 1$,
such that for all $f\in B_{s},$ , $i\in \mathbb{N}$%
\begin{equation}
\left\{ 
\begin{array}{c}
||L_{i}f||_{w}\leq ||f||_{w} \\ 
||L_{i}f||_{s}\leq \lambda _{1}||f||_{s}+B||f||_{w}.%
\end{array}%
\right.  \label{1}
\end{equation}

\item[$ML2$] There exists $M$ large enogh, such that $\lambda_1^{M}\leq \frac{1}{10(\frac{%
B}{1-\lambda_1}+1)}$ and

\begin{equation}
||L_{0}^{M}(v)||_{w}\leq \frac{1-\lambda_1}{10B} ||v||_{s}  \label{3}
\end{equation}%
for each $v\in V_s$, where%
\begin{equation*}
V_{s}=\{\mu \in B_{s}|\mu (X)=0\}.
\end{equation*}

\item[$ML3$] The family of operators is near to  $L_0$,  satisfying: $\forall i$, 
\begin{equation*}
||L_{i}-L_{0}||_{B_{s}\rightarrow B_{w}}\leq \frac{7(1-\lambda_1 )^2}{10M B({%
\frac1{1-\lambda_1}+B})}.
\end{equation*}
\end{itemize}

We remark that the assumption $(ML1)$ implies that the family of operators $%
L_{i}$ is uniformly bounded when acting on $B_{s}$ and on $B_{w}.$
Furthermore, by the Markov property,  $\forall i,$ $L_{i }(V_{s})\subseteq V_{s}$.

First, we state a Lasota-Yorke inequality for a sequential composition
of operators satisfying $(ML1)$. The statement directly follows from the definitions.

\begin{lemma}
\label{lasotaY copy(1)}Let $L_{i}$ be a family of Markov operators \
satisfying $(ML1)$ and let 
\begin{equation}
L^{(j,j+n-1)}:=L_{j+n-1}\circ  ...\circ ~L_{j}  \label{Ln}
\end{equation}%
be a sequential composition of operators in such family, then $\forall n,j$%
\begin{equation}
||L^{(j,j+n-1)}f\Vert _{w}\leq ||f\Vert _{w}
\end{equation}%
and%
\begin{equation}
||L^{(j,j+n-1)}f\Vert _{s}\leq \lambda _{1}^{n}\Vert f\Vert _{s}+\frac{B}{%
1-\lambda _{1}}\Vert f\Vert _{w}.  \label{lyw}
\end{equation}
\end{lemma}

%

The following lemma is an estimate for the distance of the sequential
composition of operators from the iterations of $L_0$.

\begin{lemma}
\label{XXX}Let $\delta \geq 0$ and let $L^{(j,j+n-1)}$ be a sequential
composition of operators $\{L_{i}\}_{i\in \mathbb{N}}$ as in $(\ref{Ln})$
that satisfies the above assumptions. Let $L_0$ as above such that $%
||L_{i}-L_0||_{s\rightarrow w}\leq \delta .$\ Then $\forall g\in
B_{s},\forall j,n\geq 1$%
\begin{equation}
||L^{(j,j+n-1)}g-L_0^{n}g||_{w}\leq \delta (\frac{1}{1-\lambda_1 }||g||_{s}+n%
\frac{B}{1-\lambda_1 }||g||_{w}).  \label{2}
\end{equation}%
where $B$ is the second coefficient of the Lasota-Yorke inequality $($\ref{1}%
$)$.
\end{lemma}

\begin{proof}
\ By the assumptions we get%
\begin{equation*}
||L_0g-L_{j}g||_{w}\leq \delta ||g||_{s}
\end{equation*}

hence the case $n=1$ of $(\ref{2})$ is trivial. \ Let us now suppose
inductively%
\begin{equation*}
||L^{(j,j+n-2)}g-L_{0}^{n-1}g||_{w}\leq \delta (C_{n-1}||g||_{s}+(n-1)\frac{B%
}{1-\lambda _{1}}||g||_{w})
\end{equation*}%
then%
\begin{eqnarray*}
||L_{j+n-1}L^{(j,j+n-2)}g-L_{0}^{n}g||_{w} &\leq
&||L_{j+n-1}L^{(j,j+n-2)}g-L_{j+n-1}L_{0}^{n-1}g
+L_{j+n-1}L_{0}^{n-1}g-L_{0}^{n}g||_{w}
\\
&\leq
&||L_{j+n-1}L^{(j,j+n-2)}g-L_{j+n-1}L_{0}^{n-1}g||_{w}+||L_{j+n-1}L_{0}^{n-1}g-L_{0}^{n}g||_{w}
\\
&\leq &\delta (C_{n-1}||g||_{s}+(n-1)\frac{B}{1-\lambda _{1}}%
||g||_{w})+||[L_{j+n-1}-L_{0}](L_{0}^{n-1}g)||_{w} \\
&\leq &\delta (C_{n-1}||g||_{s}+(n-1)\frac{B}{1-\lambda _{1}}%
||g||_{w})+\delta ||L_{0}^{n-1}g||_{s} \\
&\leq &\delta (C_{n-1}||g||_{s}+(n-1)\frac{B}{1-\lambda _{1}}||g||_{w}) \\
&&+\delta (\lambda _{1}^{n-1}||g||_{s}+\frac{B}{1-\lambda _{1}}||g||_{w}) \\
&\leq &\delta \lbrack (C_{n-1}+\lambda _{1}^{n-1})||g||_{s})+n\frac{B}{%
1-\lambda _{1}}||g||_{w}].
\end{eqnarray*}

The statement follows from the observation that, continuing the composition, 
$C_{n}$ remains bounded by the sum of a geometric series.
\end{proof}

\bigskip

\bigskip

\begin{lemma}
\label{losmem} Let $L_{i}$ be a sequence of operators satisfying $%
(ML1),...,(ML3)$. 

Then the sequence $L_{i}$ has a strong exponential loss of memory in the
following sense. There are $C,\lambda> 0$ such that $\forall j,n\in 
\mathbb{N}$, $g\in V_{s}$ 
\begin{equation*}
||L^{(j,j+n-1)}g||_{s}\leq Ce^{-\lambda n}||g||_{s}.
\end{equation*}
\end{lemma}

\begin{proof}
Remark that because of the Lasota-Yorke inequality, $\forall j,i\geq 1,g\in
B_s $ 
\begin{equation}\label{sop24}
||L^{(j,j+i)}(g)||_s\leq (\frac{B}{1-\lambda_1}+1)||g||_s.
\end{equation}
Now by $(ML2)$ let us consider $M$ such that $\lambda_1^{M}\leq \frac{1}{10(%
\frac{B}{1-\lambda_1}+1)}$  
\begin{equation*}
||{L_0}^{M}g||_{w}\leq \frac{1-\lambda_1}{10B}||g||_s.
\end{equation*}

Since 
\begin{equation*}
||L_{i}-L_0||_{s\rightarrow w}\leq \frac{7(1-\lambda_1 )^2}{10M B({\frac{1}{%
1-\lambda_1 }+B})}
\end{equation*}
for all $i $. 
By $(\ref{2})$, for every $j\in\mathbb N$ and every $i\geq M$, 
\begin{equation*}
\begin{split}
||L^{(j,j+i-1)}g-{L_0 }^{i}g||_{w} & \leq \frac{7(1-\lambda_1 )^2}{10MB(\frac{1%
}{1-\lambda_1 }+B)}(\frac{1}{1-\lambda_1 }||g||_s+i\frac{B||g||_w}{%
(1-\lambda_1 )}) \\
& \leq \frac{7i(1-\lambda_1 )}{10MB}||g ||_s.
\end{split}%
\end{equation*}
Hence 
\begin{equation}
\begin{split}
||L^{(j,j+M-1)}g||_w &\leq ||{L_0 }^{M}g||_w+ \frac{7M (1-\lambda_1 )}{10MB}||g ||_s \\
&\leq \frac{1-\lambda_1}{10B}||g||_s + \frac{7M(1-\lambda_1 )}{10MB}||g ||_s.
\end{split}%
\end{equation}

Applying now the Lasota-Yorke inequality and the estimate \eqref{sop24} we get, 
\begin{equation}\label{222}
\begin{split}
||L^{(j,j+2M-1)}g\Vert _{s} & \leq \lambda
_{1}^{M}||L^{(j,j+M-1)}g||_s+\frac{B}{1-\lambda_1}\Vert L^{(j,j+M-1)}g\Vert _{w} \\
&\leq \frac{1}{10}||g||_s+\frac{B}{1-\lambda_1} \frac{1-\lambda_1}{10B}||g||_s+\frac{7BM (1-\lambda_1 )}{(1-\lambda_1)10MB}||g ||_s \\
&\leq\frac{9}{10}||g||_s
\end{split}%
\end{equation}

Since the operators preserve zero mass, iteration
of the block contraction gives
\[
\|L^{(j,j+2kM-1)}g\|_s
\leq
\left(\frac{9}{10}\right)^k\|g\|_s
\qquad
\text{for every }j\in\mathbb N,\ k\geq1,\ g\in V_s.
\]
For an arbitrary length $n\geq1$, write
$n=2kM+r$, where $0\leq r<2M$.
The remaining block is uniformly bounded by
\[
K:=1+\frac{B}{1-\lambda_1},
\]
by the Lasota--Yorke inequality.
Consequently,
\[
\begin{aligned}
\|L^{(j,j+n-1)}g\|_s
&\leq
K\left(\frac{9}{10}\right)^k\|g\|_s
\\
&\leq
\frac{10K}{9}
\exp\left(-\frac{\log(10/9)}{2M}\,n\right)
\|g\|_s.
\end{aligned}
\]
This proves the claimed exponential loss of memory,
uniformly in the initial time $j$.
\end{proof}

{
\bigskip

We now formalize a consequence of the proof of Lemma~\ref{losmem}, which is
useful when the reference operator is not fixed, but changes slowly along the
sequence. The idea is that the proof of Lemma~\ref{losmem} is local in time:
on a block of length \(M\), the sequence only has to be close to a single
``frozen'' operator, namely the operator at the beginning of the block. The assumptions we will take and the notion of "slow enough change" will be explicitly verifiable in concrete example.

We hence formalize below the slightly different assumptions to be taken in this case. We will assume again the common Lasota--Yorke inequality \((ML1)\). We replace the
fixed convergence-to-equilibrium assumption \((ML2)\) by the following
uniform "frozen" version.

\begin{itemize}
\item[$(ML2')$] There exists \(M\geq 1\) such that
\[
\lambda_1^M
\left(
\frac{B}{1-\lambda_1}+1
\right)
\leq \frac{1}{10},
\]
and, for every \(j\in\mathbb N\) and every \(v\in V_s\),
\[
\|L_j^M v\|_w
\leq
\frac{1-\lambda_1}{10B}\|v\|_s .
\]
\end{itemize}

This means that each frozen operator \(L_j\) has an \(M\)-step convergence to
equilibrium estimate, with constants independent of \(j\).

We also assume that the sequence is slowly varying on blocks of length \(M\).

\begin{itemize}
\item[$(ML3')$] For every \(j\in\mathbb N\) and every
\(\ell=0,\ldots,M-1\),
\[
\|L_{j+\ell}-L_j\|_{B_s\to B_w}
\leq
\delta_M,
\]
where \(\delta_M>0\) is small enough so that
\[
\delta_M
\left(
\frac{1}{1-\lambda_1}
+
\frac{MB}{1-\lambda_1}
\right)
\leq
\frac{7(1-\lambda_1)}{10B}.
\]
\end{itemize}

In particular, if one assumes the stepwise slow variation estimate
\[
\|L_{n+1}-L_n\|_{B_s\to B_w}\leq \eta
\qquad \forall n,
\]
then \((ML3')\) follows whenever \(M\eta\leq \delta_M\).

By the same proof of Lemma \ref{losmem} (see \eqref{222}) one can obtain in this slightly different setting the following

\begin{lemma}[Block contraction for slowly varying compositions]
\label{lem:block-contraction-slow}
Assume \((ML1)\), \((ML2')\), and \((ML3')\). Then, for every
\(j\in\mathbb N\) and every \(g\in V_s\),
\[
\|L^{(j,j+2M-1)}g\|_s
\leq
\frac{9}{10}\|g\|_s .
\]
\end{lemma}

directly implying 
\begin{theorem}[Uniform exponential loss of memory under slow variation]
\label{thm:slow-variation-loss-memory}
Assume \((ML1)\), \((ML2')\), and \((ML3')\). Then the sequential system has
uniform exponential loss of memory. More precisely, there exist constants
\(C>0\) and \(\gamma>0\) such that, for every \(j,n\in\mathbb N\) and every
\(g\in V_s\),
\[
\|L^{(j,j+n-1)}g\|_s
\leq
C e^{-\gamma n}\|g\|_s .
\]
\end{theorem}

Showing that the argument of Lemma~\ref{losmem} does not require
the whole sequence to remain close to one fixed operator \(L_0\). It is enough
that, on each block of length \(M\), the sequence remains close to the frozen
operator at the beginning of that block, and that all frozen operators satisfy
the same \(M\)-step convergence-to-equilibrium estimate on \(V_s\).

}

\subsection{Uniformly positive kernels imply exponential
loss of memory in $L^1$}
\label{subsec:doeblin_ELoM}

Let $m$ be a probability measure on a measurable space
$(X,\mathcal A)$, and let
$K_n(x,dy)=k_n(x,y)m(dy)$, $n\in\mathbb Z$,
be Markov kernels with associated transfer operators
\[
(L_nf)(y)=\int_X f(x)k_n(x,y)\,dm(x).
\]

\begin{theorem}[Exponential loss of memory in $L^1$]
\label{thm:doeblin_ELoM}
Assume that there exists $\alpha\in(0,1]$ such that
\begin{equation}\label{eq:minorization}
K_n(x,\cdot)\geq\alpha m(\cdot)
\qquad\text{for every }n\in\mathbb Z,\ x\in X.
\end{equation}
Then, for all integers $r<s$ and all $f,g\in L^1(m)$
with $\int_X f\,dm=\int_X g\,dm$,
\begin{equation}\label{eq:ELoM_L1}
\|L_{s-1}\cdots L_r(f-g)\|_{L^1(m)}
\leq(1-\alpha)^{s-r}\|f-g\|_{L^1(m)}.
\end{equation}
In particular, the sequential system has exponential
loss of memory on the zero-mass subspace of $L^1(m)$.
\end{theorem}

\begin{proof}
Write $\Pi f=(\int_X f\,dm)\mathbf{1}$.
If $\alpha=1$, the minorization and normalization of
the kernels imply $L_n=\Pi$, so every $L_n$ annihilates
zero-mass functions.

For $0<\alpha<1$, the minorization gives
\[
L_n=\alpha\Pi+(1-\alpha)\widetilde L_n,
\qquad
\widetilde L_n:=\frac{L_n-\alpha\Pi}{1-\alpha},
\]
where $\widetilde L_n$ is positive and mass preserving,
hence an $L^1$ contraction. Therefore, whenever
$\int_X h\,dm=0$,
\[
\|L_nh\|_{L^1}
=(1-\alpha)\|\widetilde L_nh\|_{L^1}
\leq(1-\alpha)\|h\|_{L^1}.
\]
Since each $L_n$ preserves zero mass, iteration with
$h=f-g$ proves \eqref{eq:ELoM_L1}.
\end{proof}

\begin{remark}[From $L^1$ contraction to stronger
topologies via smoothing]
\label{rem:regularization_upgrade}
If, in addition, the operators satisfy a uniform
regularization estimate
\[
\|L_nf\|_{B_s}\leq R\|f\|_{L^1}
\]
for a Banach space $B_s\hookrightarrow L^1(m)$,
applying this estimate to the last step of a composition
upgrades the $L^1$ memory-loss bound to exponential
decay in $B_s$ on the zero-mass subspace.
\end{remark}

\section{Evaluation of the constants for the perturbative estimate for expanding maps}
\label{subsec:mixed_continuity_n_branches}

In this section we show an explicit estimate for the mixed norm Lipschitz constant $C(T_0)$ appearing in \ref{eq:mixed_bound_CT0}. The results of this section combined with the ones presented in Section \ref{sequential} can produce explicit examples of sequential expanding maps for which our linear response results hold.

Let $T_0,T_1:\mathbb S^1\to\mathbb S^1$ be $C^3$ \emph{expanding maps} of the same degree $n\ge 2$.
Assume $T_0$ is uniformly expanding:
\begin{equation}\label{eq:base_constants}
\lambda_0:=\inf_{x\in\mathbb S^1}|T_0'(x)|>1,\qquad
M_0:=\|T_0'\|_\infty<\infty,\qquad
M_2:=\|T_0''\|_\infty<\infty.
\end{equation}
Assume $T_1$ is $C^2$-close to $T_0$:
\begin{equation}\label{eq:C2_close_covering}
\delta:=\|T_1-T_0\|_{C^2}\le \lambda_0 -1.
\end{equation}
Since $T_1$ is assumed to be expanding, $\inf|T_1'|>1$.
Moreover, the $C^2$ bound gives
$\|T_1'\|_\infty\leq M_0+\delta$.

Choose increasing lifts
$\widetilde T_i:\mathbb R\to\mathbb R$, $i=0,1$,
satisfying
\[
\widetilde T_i(y+1)=\widetilde T_i(y)+n.
\]
The $C^2$ distance between the maps is understood
using compatible lifts, so that
\[
\|\widetilde T_1-\widetilde T_0\|_{C^2([0,1])}
=\delta.
\]
Their difference is $1$-periodic, hence the same
bounds hold on $\mathbb R$.

For $j=1,\ldots,n$, define
\[
h_{i,j}(x):=\widetilde T_i^{-1}(x+j-1),
\qquad x\in[0,1].
\]
These are real-valued $C^3$ inverse branches satisfying
\[
\widetilde T_i(h_{i,j}(x))=x+j-1.
\]
We identify functions on $S^1$ with their $1$-periodic
extensions to $\mathbb R$. Thus expressions such as
$f(h_{i,j}(x))$ and $T_i'(h_{i,j}(x))$ are well defined.
All norms and integrals in the branch variable $x$
below are taken over $[0,1]$.

The Perron--Frobenius operator $L_i$ admits the representation
\begin{equation}\label{eq:PF_covering}
(L_i f)(x)=\sum_{j=1}^n f(h_{i,j}(x))\,g_{i,j}(x),
\qquad
g_{i,j}(x):=\frac{1}{|T_i'(h_{i,j}(x))|}.
\end{equation}

\begin{lemma}[Inverse branch displacement]\label{lem:inv_branch_disp_covering_self}
For each $j=1,\dots,n$,
\[
\|h_{0,j}-h_{1,j}\|_\infty \le \frac{1}{\lambda_0}\|T_0-T_1\|_{C^0}\le \frac{\delta}{\lambda_0}.
\]
\end{lemma}

\begin{proof}
Fix $x\in[0,1]$ and set
$y_0=h_{0,j}(x)$ and $y_1=h_{1,j}(x)$.
Since
\[
\widetilde T_0(y_0)
=\widetilde T_1(y_1)=x+j-1,
\]
we have
\[
|\widetilde T_0(y_1)-\widetilde T_0(y_0)|
=
|\widetilde T_0(y_1)-\widetilde T_1(y_1)|
\leq
\|\widetilde T_0-\widetilde T_1\|_\infty.
\]
The mean value theorem on $\mathbb R$ and
$\widetilde T_0'\geq\lambda_0$ therefore give
\[
|y_1-y_0|
\leq
\frac{1}{\lambda_0}
\|\widetilde T_0-\widetilde T_1\|_\infty
\leq\frac{\delta}{\lambda_0}.
\]
Taking the supremum over $x\in[0,1]$ proves the claim.
\end{proof}


\begin{lemma}[Weight displacement]\label{lem:weight_disp_covering_self}
For each $j=1,\dots,n$,
\[
\|g_{0,j}-g_{1,j}\|_\infty
\le
\left(\frac{M_2}{\lambda_0^3}+\frac{1}{\lambda_0}\right)\delta.
\]
\end{lemma}

\begin{proof}
Fix $x\in [0,1]$ and write $y_0=h_{0,j}(x)$, $y_1=h_{1,j}(x)$. Then
\[
|g_{0,j}(x)-g_{1,j}(x)|
\le
\left|\frac{1}{|T_0'(y_0)|}-\frac{1}{|T_0'(y_1)|}\right|
+
\left|\frac{1}{|T_0'(y_1)|}-\frac{1}{|T_1'(y_1)|}\right|.
\]
For the first term, since $u\mapsto 1/u$ is $1/\lambda_0^2$--Lipschitz on $[\lambda_0,\infty)$,
\[
\left|\frac{1}{|T_0'(y_0)|}-\frac{1}{|T_0'(y_1)|}\right|
\le \frac{1}{\lambda_0^2}\,|T_0'(y_0)-T_0'(y_1)|
\le \frac{M_2}{\lambda_0^2}\,|y_0-y_1|.
\]
Using Lemma~\ref{lem:inv_branch_disp_covering_self} gives
\[
\left|\frac{1}{|T_0'(y_0)|}-\frac{1}{|T_0'(y_1)|}\right|
\le \frac{M_2}{\lambda_0^3}\|T_0-T_1\|_{C^0}\le \frac{M_2}{\lambda_0^3}\delta.
\]
For the second term, by \eqref{eq:C2_close_covering} we have $|T_0'(y_1)|\ge \lambda_0$ and $|T_1'(y_1)|\ge 1$, hence
\[
\left|\frac{1}{|T_0'(y_1)|}-\frac{1}{|T_1'(y_1)|}\right|
=
\frac{|T_1'(y_1)-T_0'(y_1)|}{|T_0'(y_1)T_1'(y_1)|}
\le \frac{\|T_1'-T_0'\|_\infty}{\lambda_0}
\leq \frac{1}{\lambda_0}\|T_0-T_1\|_{C^1}
\le \frac{1}{\lambda_0}\delta.
\]
Combining the bounds yields the claim.
\end{proof}

\begin{lemma}[Composition displacement for $W^{1,1}$]\label{lem:comp_disp_covering_self}
For each $j=1,\dots,n$ and each $f\in W^{1,1}(\mathbb S^1)$,
\[
\|f\circ h_{0,j}-f\circ h_{1,j}\|_{L^1}
\le 2(M_0+\delta)\,\|h_{0,j}-h_{1,j}\|_\infty\,\|f'\|_{L^1}
\le 2\frac{M_0+\delta}{\lambda_0}\,\delta\,\|f'\|_{L^1}.
\]
\end{lemma}

\begin{proof}
Fix a branch index $j$ and write
\[
\Delta:=\|h_{0,j}-h_{1,j}\|_\infty,
\qquad
H_t:=(1-t)h_{1,j}+t h_{0,j},
\quad t\in[0,1].
\]
The derivative bounds on $T_0$ and $T_1$ imply
\[
\frac{1}{M_0+\delta}\leq H_t'(x)\leq1.
\]
In particular, $H_t$ is increasing and its image
$H_t([0,1])$ is an interval of length at most one.

Using the absolutely continuous periodic
representative of $f$, we obtain
\[
|f(h_{0,j}(x))-f(h_{1,j}(x))|
\leq
\Delta\int_0^1 |f'(H_t(x))|\,dt.
\]
By a change of variables and periodicity,
\[
\begin{aligned}
\int_0^1 |f'(H_t(x))|\,dx
&\leq
(M_0+\delta)
\int_{H_t([0,1])}|f'(y)|\,dy
\\
&\leq
(M_0+\delta)\|f'\|_{L^1(S^1)}.
\end{aligned}
\]
Consequently,
\[
\|f\circ h_{0,j}-f\circ h_{1,j}\|_{L^1}
\leq
(M_0+\delta)\Delta\|f'\|_{L^1}.
\]
This implies the stated estimate, and Lemma~\ref{lem:inv_branch_disp_covering_self}
gives the bound in terms of $\delta$.
\end{proof}

%
%
%
%
%

\begin{proposition}[Mixed continuity $W^{1,1}\to L^1$]\label{prop:mixed_continuity_covering_self}
Suppose $\|T_0-T_1\|_{C^2}\le \lambda_0 -1$.
Then, for every $f\in W^{1,1}(\mathbb S^1)$,
\[
\|(L_0-L_1)f\|_{L^1}
\le
C(T_0)\,\|f\|_{W^{1,1}}\,\|T_0-T_1\|_{C^2},
\]
where one may take 
\[
C(T_0):=
n\left[
2\frac{M_0+\lambda_0 -1}{\lambda_0^2}
+
(M_0+\lambda_0 -1 )\left(\frac{M_2}{\lambda_0^3}+\frac{1}{\lambda_0}\right)
\right].
\]
\end{proposition}

\begin{proof}
From \eqref{eq:PF_covering},
\[
(L_0-L_1)f
=\sum_{j=1}^n \bigl(f\circ h_{0,j}-f\circ h_{1,j}\bigr)\,g_{0,j}
+
\sum_{j=1}^n (f\circ h_{1,j})\,(g_{0,j}-g_{1,j})
=:A+B.
\]
For $A$, use $\|g_{0,j}\|_\infty\le 1/\lambda_0$ and Lemma~\ref{lem:comp_disp_covering_self}:
\[
\|A\|_{L^1}\le \sum_{j=1}^n \|g_{0,j}\|_\infty\,\|f\circ h_{0,j}-f\circ h_{1,j}\|_{L^1}
\le 2n\cdot \frac{1}{\lambda_0}\cdot \frac{M_0+\delta}{\lambda_0}\,\delta\,\|f'\|_{L^1}.
\]
Since $\delta\le \lambda_0 -1$, this gives
\[
\|A\|_{L^1}\le 2n\cdot \frac{M_0+\lambda_0 -1}{\lambda_0^2}\,\delta\,\|f'\|_{L^1}.
\]
For $B$, by change of variables $x=\widetilde T_1(y) -j+1$ on the image of $h_{1,j}$,
\[
\begin{aligned}
\|f\circ h_{1,j}\|_{L^1}
&=
\int_0^1 |f(h_{1,j}(x))|\,dx
\\
&=
\int_{h_{1,j}([0,1])}
|f(y)|\,\widetilde T_1'(y)\,dy
\\
&\leq
(M_0+\delta)\|f\|_{L^1(S^1)}.
\end{aligned}
\]
Here we used periodicity of $f$ and the fact that
$h_{1,j}([0,1])$ has length at most one.
%
%
Thus, using Lemma~\ref{lem:weight_disp_covering_self},
\[
\|B\|_{L^1}
\le \sum_{j=1}^n \|f\circ h_{1,j}\|_{L^1}\,\|g_{0,j}-g_{1,j}\|_\infty
\le n\,(M_0+\delta)\left(\frac{M_2}{\lambda_0^3}+\frac{1}{\lambda_0}\right)\delta\,\|f\|_{L^1}.
\]
Again using $\delta\le\lambda_0-1$ and $\|f\|_{W^{1,1}}=\|f\|_{L^1}+\|f'\|_{L^1}$ yields the stated bound.
\end{proof}

\section*{Acknowledgements}
We wish to thank A. Vulpiani and M. Branicki for stimulating conversations. VL acknowledges the partial support provided by the Horizon Europe Projects Past2Future (Grant No. 101184070) and ClimTIP (Grant No. 100018693), by the ARIA SCOP-PR01-P003—Advancing Tipping Point Early Warning AdvanTip project, and by the European Space Agency Project PREDICT (Contract 4000146344/24/I-LR).

SG acknowledges the MIUR Excellence Department Project awarded to the
Department of Mathematics, University of Pisa, CUP I57G22000700001.

\section*{Data Availability Statement}

No datasets were generated or analysed during the current study.

\section*{Conflict of interest} The authors declare that they have no conflict of interest.

\bibliographystyle{unsrt}  
\bibliography{biblio}

\end{document}